\documentclass[12pt]{article}
\usepackage{amsmath,amssymb,amsfonts,amsthm}
\usepackage{graphicx}
\usepackage{color}
\newtheorem{theorem}{Theorem}
\numberwithin{theorem}{section}
\newtheorem{corollary}[theorem]{Corollary}
\newtheorem{lemma}[theorem]{Lemma}
\newtheorem{proposition}[theorem]{Proposition}
\newtheorem{conjecture}[theorem]{Conjecture}

\newtheorem{definition}[theorem]{Definition}
\theoremstyle{remark}
\def\eref#1{(\ref{#1})}

\newcommand{\R}{\mathbb{R}}
\newcommand{\C}{\mathbb{C}}
\newcommand{\D}{\mathbb{D}}
\newcommand{\Z}{\mathbb{Z}}

\def\H{\mathbb{H}}

\def\Im{{\rm Im}\,}
\def\Re{{\rm Re}\,}
\def\SLEkk#1/{$\mathrm{SLE}_{#1}$}
\def\SLEk/{\SLEkk{\kappa}/}
\def\SLEtwo/{\SLEkk2/}
\def\SLE/{$\mathrm{SLE}$}
\def\CLEkk#1/{$\mathrm{CLE}_{#1}$}
\def\CLEk/{\CLEkk{\kappa}/}
\def\CLEtwo/{\CLEkk2/}
\def\CLE/{$\mathrm{CLE}$}
\def\GLEkk#1/{$\mathrm{GLE}_{#1}$}
\def\GLEk/{\GLEkk{\kappa}/}
\def\GLEtwo/{\GLEkk2/}
\def\GLE/{$\mathrm{GLE}$}

\def\Ito/{It\^o}
\def \eps {\varepsilon}

\def \proof {{ \medbreak \noindent {\bf Proof.} }}
\def\proofof#1{{ \medbreak \noindent {\bf Proof of #1.} }}

\def\density{\rho}

\def\Var{\mathrm{Var}}
\def\Cov{\mathrm{Cov}}

\def\GFFSurvey{MR2322706}

\title{Liouville Quantum Gravity and KPZ}
\author{{\sc Bertrand Duplantier\thanks{e-mail: \texttt{Bertrand.Duplantier@cea.fr}.  Partially supported by  grant  ANR-08-BLAN-0311-CSD5 and CNRS grant PEPS-PTI 2010.}}\, and
{\sc Scott Sheffield\thanks{e-mail: \texttt{sheffield@math.mit.edu}.
Partially supported by NSF grants DMS 0403182 and DMS 064558 and
OISE 0730136.}}
\\
{\it $^*$Institut de Physique Th\'{e}orique, CEA/Saclay}\\
{\it F-91191 Gif-sur-Yvette Cedex, France}\\
{\&}\\
{\it $^\dagger$Department of Mathematics}\\
{\it Massachusetts Institute for Technology}\\
{\it Cambridge, Massachusetts 02139, USA}}

\date{November 19, 2010}

\begin{document}
\maketitle
\begin{abstract}
Consider a bounded planar domain $D$, an instance $h$ of the
Gaussian free field on $D$, with Dirichlet energy $({2\pi})^{-1}
\int_D \nabla h(z) \cdot \nabla h(z)dz,$ and a constant $0\leq
\gamma <2$.  The {\bf Liouville quantum gravity measure} on $D$ is
the weak limit as $\varepsilon \to 0$ of the measures
$$\varepsilon^{\gamma^2/2} e^{\gamma h_\varepsilon(z)}dz,$$ where
$dz$ is Lebesgue measure on $D$ and $h_\varepsilon(z)$ denotes the
mean value of $h$ on the circle of radius $\varepsilon$ centered at
$z$.  Given a random (or deterministic) subset $X$ of $D$ one can
define the scaling dimension of $X$ using either Lebesgue measure or
this random measure.  We derive a general quadratic relation between
these two dimensions, which we view as a probabilistic formulation
of the Knizhnik, Polyakov, Zamolodchikov (KPZ, 1988) relation from conformal field theory.  We also present a
boundary analog of KPZ (for subsets of $\partial D$). We discuss the connection between
discrete and continuum quantum gravity and provide a framework for
understanding Euclidean scaling exponents via quantum gravity.
\end{abstract}

\newpage

\begin{quotation} {\it ``There are methods and formulae in science, which serve as master-keys to
many apparently different problems.  The resources of such things
have to be refilled from time to time.  In my opinion at the present
time we have to develop an art of handling sums over random
surfaces.  These sums replace the old-fashioned (and extremely
useful) sums over random paths. The replacement is necessary,
because today gauge invariance plays the central role in physics.
Elementary excitations in gauge theories are formed by the flux
lines (closed in the absence of charges) and the time development of
these lines forms the world surfaces.  All transition amplitude[s]
are given by the sums over all possible surfaces with fixed
boundary.''} (A.M. Polyakov, Moscow, 1981.) \cite{MR623209}
\end{quotation}

\section{Introduction}
\subsection{Overview}
The study of certain natural probability measures on the space of
two dimensional Riemannian manifolds (and singular limits of these
manifolds) is often called ``two-dimensional quantum gravity.''
These models have been very thoroughly studied in the physics
literature, in part because of connections to string theory and
conformal field theory \cite{MR623209, MR623210, MR1122810,
MR1052937,1990PThPS.102..319S,Ginsparg-Moore, MR1307402,MR1461284,
MR1360409, MR1338099, MR1320471, 1995PhRvD..51.1836K,
1996NuPhS..45..135K, MR1465433,Eynard2000, 2006math.ph...8053D}, and
to random matrix theory and geometrical models; see, e.g., the
references \cite{1978CMaPh..59...35B,
1985NuPhB.257..433A,1985PhLB..157..295K,1985PhLB..159..303D,1986NuPhB.275..641B,
1986PhLB..174...87B,1986PhLA..119..140K,
1988PhRvL..61.1433D,1990NuPhB.340..491D,1989PhLB..220..200G,1989MPLA....4..217K,1989NuPhB.326..583K,1990MPLA....5.1041D,1991NuPhB.362..665M,
1992NuPhB.386..520K,1992NuPhB.386..558E,1992PhLB..286..239J,1992PhLB..296..323K,1992MPLA....7.3081K,1993NuPhB.394..383A,1994NuPhB.426..203D,1994MPLA....9.1221A,
Daul-1995,
1995NuPhB.455..577E,1995NuPhB.434..264K,1995NuPhB.440..189B,1996PhLB..388..713A,MR1666816,
 1999HMDup, MR1723364, 1999PhRvL..82..880D,Eynard-Bonnet,MR1690386,MR1762323,2000PhRvL..84.1363D,
MR1749396, 2002PhRvL..89z4101D,MR2112128,2007JSMTE..08...23K,2008ExactMethods}. More recently, a purely combinatorial
approach to discretized quantum gravity has been successful
\cite{schaeffer1998, flajolet1, flajolet2, MR1938319, 2002math.....11070B, MR2013797,
MR1965114, MR1987861, MR2152580, MR2369957, MR2571957,
MR2336042, MR2349571,bernardi-2006,2006math......8057B,2006math......1678B,2006math.....12003B,2008JPhA...41n5001B, MR2375600, MR2399286,
2008JSMTE..07..020B,2008LeGall,2009JSMTE..03..001B,2009arXiv0907.3262L,2009arXiv0909.1695B}, {as well as the so-called topological expansion involving higher-genus random surfaces
\cite{MR2563085,MR2507734,2008arXiv0804.0546C,2007math.ph...2045E,
2008JHEP...06..037E,2009JHEP...03..003E}.}

One of the most influential papers in this field is a 1988 work of
Knizhnik, Polyakov, and Zamolodchikov \cite{MR947880}.  Building on
a 1987 work of Polyakov \cite{Polyakov:1987zb}, the authors derive a
relationship (the {\bf KPZ formula}) between scaling dimensions of
fields defined using Euclidean geometry and analogous dimensions
defined via Liouville quantum gravity (as described earlier in
 \cite{MR623209, MR623210}; {see \cite{2008arXiv0812.0183} for a recent historical recount}). An alternative heuristic derivation
using Liouville field theory in the so-called conformal gauge was
proposed shortly after \cite{MR981529, MR1005268} (see also
\cite{1993MPLA....8.3529T}). The original work by KPZ has been cited
roughly a thousand times in a variety of contexts, which we will not
attempt to survey here, though we mention that there have been a
number of explicit calculations in Liouville field theory with
matching results in the random matrix theory approach, e.g.,
\cite{1991PhRvL..66.2051G,1994NuPhB.429..375D,1995PhLB..363...65T,1996NuPhB.477..577Z,fateev-2000,
MR1867860,2001JHEP...11..044H,MR1877816,2003NuPhB.658..397K,Zamolo2004,MR2057108,
2006CMaPh.268..135T,MR2354665}; {for a review, see \cite{2004IJMPA..19.2771N}.}


The relationship in \cite{MR947880} has never been proved or even
precisely formulated mathematically.  The main goal of this work is
to formulate and prove the KPZ scaling dimension relationship in a
probabilistic setting.

\subsection{Critical Liouville quantum gravity}

The study of two dimensional random surfaces makes frequent use of
the Riemann uniformization theorem, which states that every smooth
simply connected Riemannian manifold $\mathcal M$ can be conformally
mapped to either the unit disc $\D$, the complex plane $\C$, or the
complex sphere $\C \cup \{\infty\}$. (If a manifold is not simply
connected then its universal cover can be conformally mapped to one
of these spaces. See, e.g., Chapter 4 of \cite{MR1139765} for more
exposition; see also \cite{MR2165685, MR2165683, MR2131879,
MR1958012, MR2339977} for approximation algorithms and beautiful
computer illustrations of these maps.) Another way to say this
is that $\mathcal M$ can be parameterized by points $z = x+iy$ in
one of these spaces in such a way that the metric takes the form
$e^{\lambda(z)}(dx^2 + dy^2)$ for some real-valued function
$\lambda$.  The $(x,y)$ are called {\em isothermal coordinates} or
{\em isothermal parameters} for $\mathcal M$.  In most of this paper
we let the parameter space be a general simply connected proper
subdomain $D$ of the plane (which, of course, is conformally
equivalent to $\D$).

We remark that the existence of isothermal coordinates does not
require that $\mathcal M$ be smooth; for example, it can be deduced
whenever $\mathcal M$ can be parameterized by a simply connected
planar domain in which the metric has the form $E(x,y)dx^2 +
2F(x,y)dxdy + G(x,y)dy^2$ where $EG-F^2 > 0$, $E>0$, and $E$, $F$,
and $G$ are $\beta$-H\"older continuous for some $0 < \beta < 1$
\cite{MR0074856}.

Length, area, and curvature are easy to express in isothermal
coordinates.  The length of a path in $\mathcal M$ parameterized by
a smooth path $P$ in $D$ is given by
$$\int_P e^{\lambda(s)/2} ds,$$ where $ds$ is
the Euclidean length measure on $D$. Given a measurable subset $A$
of $D$, the integral $\int_A e^{\lambda(z)} dz$ (where $dz$ denotes
Lebesgue measure on $D$) is the area of the portion of $\mathcal M$
parameterized by $A$.  The function $K = - e^{-\lambda} \Delta
\lambda$ (where $\Delta \lambda = \lambda_{xx} + \lambda_{yy}$ is
the Laplacian operator) is called the {\bf Gaussian curvature} of
$\mathcal M$.  If $A$ is a measurable subset of the $(x,y)$
parameter space, then the integral of the Gaussian curvature with
respect to the portion of $\mathcal M$ parameterized by $A$ can be
written $\int_A e^{\lambda(z)} K(z) dz = \int_A -\Delta \lambda(z)
dz$ where $dz$ denotes Lebesgue measure on $D$. In other words,
$-\Delta \lambda$ gives the density of Gaussian curvature in the
isothermal coordinate space.  In particular, $\mathcal M$ is flat if
and only if $\lambda$ is harmonic.


The above suggests that one can study random simply connected
Riemannian manifolds by studying random functions $\lambda$ on $\C$
or $\C \cup \{\infty\}$ or any fixed simply connected subdomain $D$
of $\C$. In the probabilistic formulation of the so-called critical
Liouville quantum gravity, $\lambda$ is taken to be a multiple of
the Gaussian free field (GFF), although some care will be required
to make sense of this construction, since the GFF is a distribution
and not a function.  (The relationship between our probabilistic
formulation and the original formulation of Polyakov will be
discussed in Section \ref{s.coordchange}.)

For concreteness, let $h$ be an instance of a centered GFF on a
bounded simply connected domain $D$ with zero boundary conditions.
This means that $h = \sum_n \alpha_n f_n$ where the $\alpha_n$ are
i.i.d.\ zero mean unit variance normal random variables and the
$f_n$ are an orthonormal basis, with respect to the inner product
$$(f_1, f_2)_\nabla := (2\pi)^{-1} \int_D \nabla f_1(z) \cdot \nabla f_2(z) dz,$$
of the Hilbert space closure $H(D)$ of the space $H_s(D)$ of
$C^\infty$ real-valued functions compactly supported on $D$.
Although this sum diverges pointwise almost surely, it does converge
almost surely in the space of distributions on $D$, and one can also
make sense of the mean value of $h$ on various sets. (See
\cite{\GFFSurvey} for a detailed account of this construction of the
GFF; see Section \ref{ss.GFFnormsection} for a quick overview. Note
that the $(2 \pi)^{-1}$ in the definition above does not appear,
e.g., in \cite{\GFFSurvey}; including this factor in the definition,
as is common in the physics literature, is equivalent to multiplying
the corresponding $h$ by $\sqrt{2\pi}$. This will simplify some of
our formulas later on.  In particular, in this formulation the two
point covariance scales like $-\log(|z-w|)$ instead of $-(2\pi)^{-1}
\log(|z-w|)$; see Section \ref{ss.GFFnormsection}.)

Given an instance $h$ of the Gaussian free field on $D$, let
$h_\varepsilon(z)$ denote the mean value of $h$ on {$\partial B_\varepsilon(z)$,} the circle of
radius $\varepsilon$ centered at $z$ (where $h(z)$ is defined to be
zero for $z \in \C \setminus D$). This is almost surely a locally
H\"older continuous function of $(\varepsilon, z)$ on $(0,\infty)
\times D$ (see Section \ref{ss.GFFnormsection}). For each fixed
$\varepsilon$, consider the surface $\mathcal M_\varepsilon$
parameterized by $D$ with metric $e^{\gamma h_\varepsilon(z)}(dx^2 +
dy^2)$.  We would like to define a surface $\mathcal M$
parameterized by $D$ to be some sort of limit as $\varepsilon \to 0$
of these surfaces. Since we would not expect the limit to be a
Riemannian manifold in any classical sense, we have to state
carefully what we mean by this. There are many ways we could attempt
to make sense of this limit, depending on what quantities we focus
on. For example, we could consider
\begin{enumerate} \item The length of the shortest path connecting a
fixed pair of points in $D$.
\item The area of a fixed subset of $D$.
\item The length of a fixed smooth curve in $D$.
\item The length of a smooth boundary arc of $D$ (which becomes interesting
when $h$ is an instance of the GFF with free boundary conditions).
\end{enumerate}
Intuitively, we might expect each quantity above to scale like a
random constant times a (possibly different) power of $\varepsilon$
as $\varepsilon$ tends to zero --- i.e., we would expect that if the
$\mathcal M_\varepsilon$ were rescaled by the appropriate powers of
$\varepsilon$, the above quantities would have limits as
$\varepsilon \to 0$. Focusing on lengths of shortest paths, one
might guess that the random surfaces $\mathcal M_\varepsilon$
(rescaled by some power of $\varepsilon$) would almost surely
converge (in some natural topology on the set of metric spaces) to a
non-trivial random metric space parameterized by $D$. However, this
is not something we are currently able to prove. Focusing on areas,
one might expect that for some $\alpha$ the renormalized area
measures $\varepsilon^\alpha e^{\gamma h_\varepsilon(z)}dz$ would
almost surely converge weakly to a random measure on $D$.  This is
the limit we will construct and work with in this paper. We will
also address the lengths of fixed curves and boundary curves; see
Section \ref{boundaryKPZsection}. Although the constructions are
quite similar, we will not use the so-called Wick normal ordering
terminology in this paper (see e.g., \cite{BSimon}).  We present a
self-contained proof of the following (although similar measures
have appeared much earlier, and are called the {\em H{\o}egh-Krohn
model} \cite{MR0292433}
--- see also \cite{MR553970,MR0356761} for a discussion on the level
of Schwinger functions, and a more recent survey \cite{MR1168301}):

\begin{proposition} \label{p.hepsilonlimit}
Fix $\gamma \in [0, 2)$ and define {the zero boundary GFF $h$} and $D$ as above. Then it is
almost surely the case that as $\varepsilon \to 0$ along powers of
two, the measures $\mu_\varepsilon := \varepsilon^{\gamma^2/2}
e^{\gamma h_\varepsilon(z)}dz$ converge weakly {inside $D$} to a limiting
measure, which we denote by $\mu = \mu_h = e^{\gamma h(z)}dz$.
This remains true if we replace $h$ with a non-centered GFF on $D$
--- i.e., if we set $h = \overline h+h^0$ where $\overline h$ is the
zero boundary GFF on $D$ and $h^0$ is a deterministic, non-zero
continuous function on $D$.
\end{proposition}

For each $z \in D$, denote by $C(z;D)$ the conformal radius of $D$
viewed from $z$.  That is, $C(z;D) = |\phi'(z)|^{-1}$ where $\phi:D
\to \D$ is a conformal map to the unit disc with $\phi(z) = 0$.  The
following gives an equivalent definition of $\mu$.

\begin{proposition} \label{p.hnlimit}
Write $h = \overline h +h^0$ where $\overline h$ is the zero
boundary GFF on $D$ and $h^0$ is a deterministic continuous function
on $D$. Let $f_1, f_2, \ldots$ be an orthonormal basis for $H(D)$
comprised of continuous functions on $D$ and let $h^n$ be the
expectation of $h$ given its projection onto the span of $\{f_1,
f_2, \ldots, f_n\}$.  (In other words, $h^n$ is $h^0$ plus the
projection of $\overline h$ onto the span of $\{f_1, f_2, \ldots,
f_n\}$.)  Then $\mu = \mu_h$ (as defined in Proposition
\ref{p.hepsilonlimit}) is almost surely the weak limit for $n\to + \infty$ of the
measures
\begin{equation} \label{e.hn} \mu^n = \exp \left( \gamma h^n(z) -
\frac{\gamma^2}{2}\Var h^n(z) + \frac{\gamma^2}{2}\log C(z; D)
\right)dz.\end{equation} For each measurable $A \subset D$, we have
\begin{equation}\label{condexpect}
\mathbb E [ \mu(A) | h^n] = \mu^n(A).
\end{equation} In particular,
$$\mathbb E \mu(A) = \int_A C(z;D)^{\frac{\gamma^2}{2}}e^{\gamma
h^0(z)}dz.$$
\end{proposition}

Intuitively, we interpret the pair $(D, \mu)$ as describing a
``random surface'' $\mathcal M$ parameterized conformally by $D$,
with area measure given by $\mu$.  {In the physics literature, the more
commonly used term is ``random metric''; however, we stress that we have not
endowed $D$ with a two point distance function, so we cannot mathematically interpret
``random metric'' to mean ``random metric space.''}

{In the Liouville quantum gravity literature, the term ``metric'' is used to mean alternately
a two-point distance function, a measure of areas and lengths of
curves, or a Riemannian metric tensor (usually the latter).  The first maps pairs of points to $\R^+$, the second maps sets/curves to $\R^+$, and the third maps pairs of tangent vectors to $\R$.
A smooth manifold can be equivalently
characterized by any one of these objects; however, the relationships between
these notions are less obvious for the limiting (and highly non-smooth) ``random surfaces'' $\mathcal M$ we deal with here.  The pair $(D,\mu)$ represents a conformal
parameterization of $\mathcal M$, with area measure $\mu$.  However,
further work would be required to use this structure to construct a two-point distance function on $\mathcal M$, or vice versa.  To avoid ambiguity arising from the multiple definitions
of the term ``metric'', we will use the term ``random surface'' instead of ``random metric'' in this paper to describe the pair $(D,\mu)$.}

\subsection{Scaling exponents and KPZ}

\begin{definition} \label{Bdelta} For any fixed measure $\mu$ on $D$ (which we call the ``quantum''
measure), we let $B^\delta(z)$ be the Euclidean ball centered at $z$
whose radius is chosen so that $\mu(B^\delta(z)) = \delta$.  (If
there does not exist a unique $\delta$ with this property, take the
radius to be $\sup \{ \varepsilon: \mu(B_\varepsilon(z)) \leq \delta
\}$.)
We refer to $B^\delta(z)$ as the {\bf isothermal quantum ball}
of area $\delta$ centered at $z$. In particular, if $\gamma = 0$
then $\mu$ is Lebesgue measure and $B^\delta(z)$ is
$B_{\varepsilon}(z)$ where $\delta = \pi \varepsilon^2$.
\end{definition}
Given a subset $X \subset D$, we denote the $\varepsilon$
neighborhood of $X$ by
$$B_\varepsilon(X) = \{ z : B_\varepsilon(z) \cap X \not = \emptyset \}.$$
We also define the {\bf isothermal quantum $\delta$ neighborhood} of
$X$ by
$$B^\delta(X) = \{z : B^\delta(z) \cap X \not = \emptyset \}.$$

Translated into probability language, the so-called {\bf KPZ
formula} is a quadratic relationship between the expectation fractal
dimension of a random subset of $D$ defined in terms of Euclidean
measure (which is the Liouville gravity measure with $\gamma = 0$)
and the corresponding expectation fractal dimension of $X$ defined
in terms of Liouville gravity with $\gamma \not = 0$.

Fix $\gamma \in [0,2)$ and let $\mu_0$ denote Lebesgue measure on
$D$.  We say that a (deterministic or random) fractal subset $X$ of
$D$ has {\bf Euclidean expectation dimension} $2-2x$ and {\bf
Euclidean scaling exponent $x$} if the expected area of
$B_\varepsilon(X)$ decays like $\varepsilon^{2x} =
(\varepsilon^2)^x$, i.e.,
$$\lim_{\varepsilon \to 0} \frac{ \log \mathbb E \mu_0(B_\varepsilon(X)) }{\log \varepsilon^2} = x.$$
We say that $X$ has {\bf quantum scaling exponent $\Delta$} if when
$X$ and $\mu$ (as defined above) are chosen independently we have
$$\lim_{\delta \to 0} \frac{ \log \mathbb E \mu(B^\delta(X)) }{\log \delta} = \Delta,$$
{where here $\mathbb E$ is with respect to both random variables $X$ and $\mu$.}
(Section \ref{discreteheuristicoverview} will provide some discrete
quantum gravity heuristics that motivate the idea of taking $X$ and
$\mu$ to be independent of one another, as well as our particular definition
of scaling exponent.)

The following is the KPZ scaling exponent relation. To avoid
boundary technicalities, we restrict attention here to a compact
subset of $D$. The case of boundary exponents will be dealt with in
Section \ref{boundaryKPZsection}.

\begin{theorem} \label{t.strongquantumKPZ}
Fix $\gamma \in [0,2)$ and a compact subset $\tilde D$ of $D$.  If
$X \cap \tilde D$ has Euclidean scaling exponent $x \geq 0$ then it
has quantum scaling exponent $\Delta$, where $\Delta$ is the
non-negative solution to
\begin{equation} \label{KPZ3} x = \frac{\gamma^2}{4} \Delta^2 + \left( 1 -
\frac{\gamma^2}{4}\right)\Delta.
\end{equation}
\end{theorem}

It also turns out that Theorem \ref{t.strongquantumKPZ} admits the
following straightforward generalization:

\begin{theorem} \label{t.verystrongquantumKPZ}
Let $\mathcal X$ be any random measurable subset of the set of all
balls of the form $B_\varepsilon(z)$ for $\varepsilon>0$ and $z$ in
a fixed compact subset $\tilde D$ of $D$.  Fix $\gamma \in [0,2)$.
Then if
$$\lim_{\varepsilon \to 0} \frac{ \log \mathbb E \mu_0 \{z: B_\varepsilon(z) \in \mathcal X\} }{\log \varepsilon^2} = x,$$
then it follows that, when $\mathcal X$ and $\mu$ (as defined above)
are chosen independently, we have
$$\lim_{\delta \to 0} \frac{ \log \mathbb E \mu \{z: B^\delta(z) \in \mathcal X\} }{\log \delta} = \Delta,$$
where $\Delta$ is the non-negative solution to $$ x =
\frac{\gamma^2}{4} \Delta^2 + \left( 1 -
\frac{\gamma^2}{4}\right)\Delta.
$$
\end{theorem}
(Again, expectation in the above theorem is with respect to both
random variables, $X$ and $\mu$.) We obtain Theorem
\ref{t.strongquantumKPZ} as a special case of Theorem
\ref{t.verystrongquantumKPZ} by writing $\mathcal X = \{
B_\varepsilon(z): B_\varepsilon(z) \cap X \not = \emptyset \}.$
Theorem \ref{t.verystrongquantumKPZ} allows us to consider $x$ that
are greater than $1$ (in which case the ``dimension'' $2-2x$ would
be negative). If one considers, for example, a conformal loop
ensemble on $D$ with $\kappa = 6$ (corresponding to a scaling limit
of the cluster-boundary loops in site percolation on the triangular
lattice) one could let $\mathcal X$ be the set of balls contained in
$\tilde D$ that intersect $\ell$ distinct ``macroscopic'' loops
(where ``macroscopic'' means that their diameters are greater than
some fixed constant). In this case, the value $x$ depends on $\ell$
and is called a {\bf multi-arm exponent}
\cite{1987PhRvL..58.2325S,1999HMDup,1999PhRvL..83.1359A,2001math......9120S}
and we may view the corresponding $\Delta$ as a quantum analog of
such an exponent.

As another example, for some integer $L$ fix distinct points
$z_1,z_2, \ldots, z_L$ in $D \setminus \tilde D$ and run $L$
independent Brownian motions started at the points $z_1, \ldots,
z_L$.  Then let $\mathcal X$ be the set of balls $B_\varepsilon(z)$
contained in $\tilde D$ with the property that the Brownian motions
--- stopped at the first time they intersect $\partial
B_\varepsilon(z)$
--- do not intersect one another.

In this case, the Euclidean scaling exponent $x=x_L$ is called a {\bf
Brownian intersection exponent}.  It was conjectured in \cite{1988DupKwon} and rigorously
derived in a celebrated series of papers by Lawler, Schramm, and
Werner using the Schramm-Loewner evolution with $\kappa = 6$
\cite{MR2002m:60159a, MR2002m:60159b, MR1899232}:
$$x_L = \frac{1}{24}(4L^2-1).$$
Although we will not fully explain this in this paper, there is a
close connection between \SLEk/ and Liouville quantum gravity models
with $\gamma = \sqrt{\min\{\kappa, 16/\kappa \}}$ (see Section
\ref{discreteheuristicoverview}), in agreement with the relationship between CFT
central charge $c$ and parameter $\gamma$ in Liouville quantum
gravity
\cite{MR947880,MR981529,MR1005268,1990PThPS.102..319S,Ginsparg-Moore}.
Taking $\gamma = \sqrt{16/6} = \sqrt{8/3}$ and $x_L$ as above, the
KPZ formula gives
$$\Delta_L = \frac{1}{2} \left(L-\frac{1}{2}\right),$$ which is an
affine function of $L$. The first co-author predicted several years
ago, based on an approach via discrete quantum gravity models, that
this $\Delta$ would be an affine function of $L$ (see
\cite{MR1666816,MR1723364,1999PhRvL..82..880D,MR2112128} and the
discussion in Section \ref{discreteheuristicoverview}).  The
derivation is based on a simple and general geometric argument that
discrete quantum gravity exponents should be in a certain sense
additive together with a heuristic connection between the discrete
and the continuous models. A direct calculation via discrete graphs
appears in \cite{MR1666816}.  This is related to the cascade
relations given earlier by Lawler and Werner using different
techniques \cite{MR1742883}.

Three papers that build on our work (as announced and presented in
talks and minicourses beginning in 2007, and later  in the Letter
\cite{2009arXiv0901.0277D}) have already been posted
online: Benjamini and Schramm cited the ideas of our paper to
produce an analog of Theorem \ref{t.strongquantumKPZ} in a one
dimensional cascade model; their proof uses a Frostman measure
construction in place of the large deviations construction used
here, and almost sure Hausdorff dimension in place of expectation
dimension \cite{benjamini-2008}.  A follow up paper
\cite{rhodes-2008} adapts the arguments of \cite{benjamini-2008} to
a class of cascade models, which was expanded to include (in a
revised version) a measure based on the exponential of the Gaussian
free field, like the measures we construct here. Another paper provides a
heuristic heat kernel based derivation of the KPZ relation \cite{2008arXiv0810.2858D}.

Intuitively, one reason to expect Hausdorff-like variants of KPZ to
be accessible is that the second moments (and higher moments) of the
random measures are essentially trivial to compute (see Section
\ref{ss.randommetrics}).  It might be interesting to try to derive
other variants of KPZ: for example, one could try to relate the
actual Minksowki or Hausdorff measure of a set, in the Euclidean
sense, with some kind of expected Minkowski or Hausdorff measure in
the quantum sense. We will not address these alternative
formulations here. However, we will present below a picturesque
formulation of KPZ in terms of box decompositions.

\subsection{Statement of box formulation of KPZ}

Define a {\bf diadic square} to be a closed square (including its
interior) of one of the grids $2^{-k} \mathbb Z^2$ for some integer
$k$. Let $\mu$ be any measure on $\C$.  For $\delta > 0$, we define a {\bf
$(\mu,\delta)$ box} $S$ to be a diadic square $S$ with $\mu(S) <
\delta$ and $\mu(S') \geq \delta$ where $S'$ is the diadic parent of
$S$.  Clearly, if a point $z\in \C$ does not lie on a boundary of a
diadic square---and it satisfies $\mu(\{z\}) < \delta <
\mu(\C)$---then there is a unique $(\mu,\delta)$ box containing $z$,
which we denote by $S^\delta(z)$. Let $\mathcal S_\mu^\delta$ be the
set of all $(\mu,\delta)$ boxes. The boxes in $\mathcal
S_\mu^\delta$ do not overlap one another except at their boundaries.
Thus, they form a tiling of $\R^2$ (see Figures \ref{f.QG1},
\ref{f.QG2}, and \ref{f.QG3} for an illustration of this
construction on a torus).

We remark that the $(\mu,\delta)$ boxes should not be confused with
the diadic boxes in the so-called $\delta$-{\em Calder\'on Zygmund
decomposition} of $\mu$. Readers familiar with that decomposition
may recall that while the $(\mu,\delta)$ boxes are diadic squares
$S$ with $\mu(S) < \delta \leq \mu(S')$, the $\delta$-Calder\'on
Zygmund boxes are diadic squares $S$ with $\mu(S)/\mu_0(S)
> \delta \geq \mu(S')/\mu_0(S')$, where $\mu_0$ is Lebesgue
measure.  Roughly speaking, the $\mu$ {\em measure} on each
$(\mu,\delta)$ box approximates $\delta$, while the $\mu$ {\em
density} on each Calder\'on Zygmund box approximates $\delta$.

\begin{figure}
\begin{center}
\includegraphics[scale = 1.6]{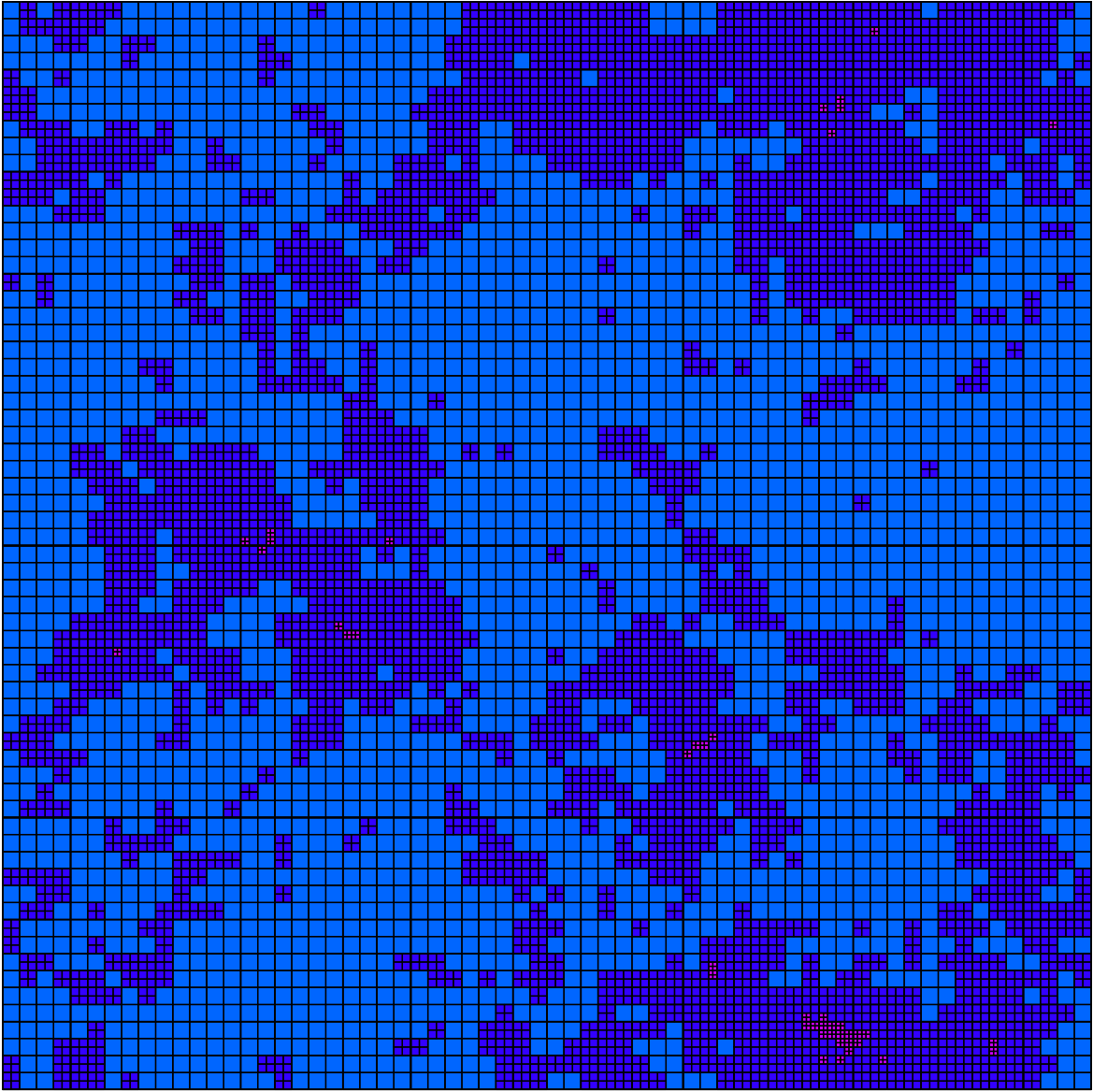}
\end{center}
\caption{\label{f.QG1} $(\mu, \delta)$ boxes of the random measure
$\mu = e^{\gamma h} dz$, where $\gamma=.5$ and $h$ is the (discrete)
Gaussian free field on a very fine ($1024 \times 1024$) grid on the
torus, $dz$ is counting measure on the vertices of that grid, and
$\delta$ is $2^{-12}$ times the total mass of $\mu$. (We view $\mu$
as an approximation of the continuum Liouville quantum gravity
measure.) One way to construct this figure is to view the entire
torus as a square; then subdivide each square whose $\mu$ measure is
at least $\delta$ into four smaller squares, and repeat until all
squares have $\mu$ measure less than $\delta$. The squares shown
have roughly the same $\mu$ size --- in the sense that each square
has $\mu$ measure less than $\delta$ but each square's diadic parent
has $\mu$ measure greater than $\delta$.}
\end{figure}

\begin{figure}
\begin{center}
\includegraphics[scale = 1.6]{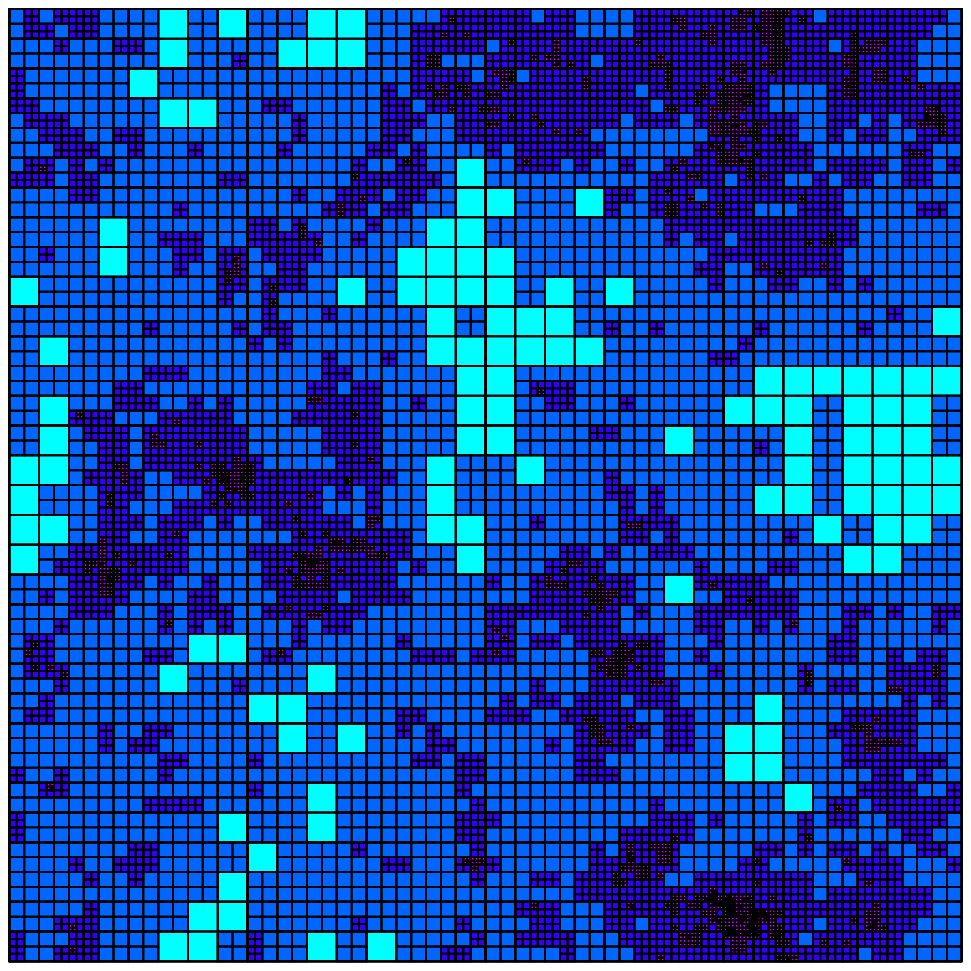}
\end{center}
\caption{\label{f.QG2} Analog of Figure \ref{f.QG1} with $\gamma =
1$, using the same instance $h$ of the GFF.}
\end{figure}

\begin{figure}
\begin{center}
\includegraphics[scale = 1.6]{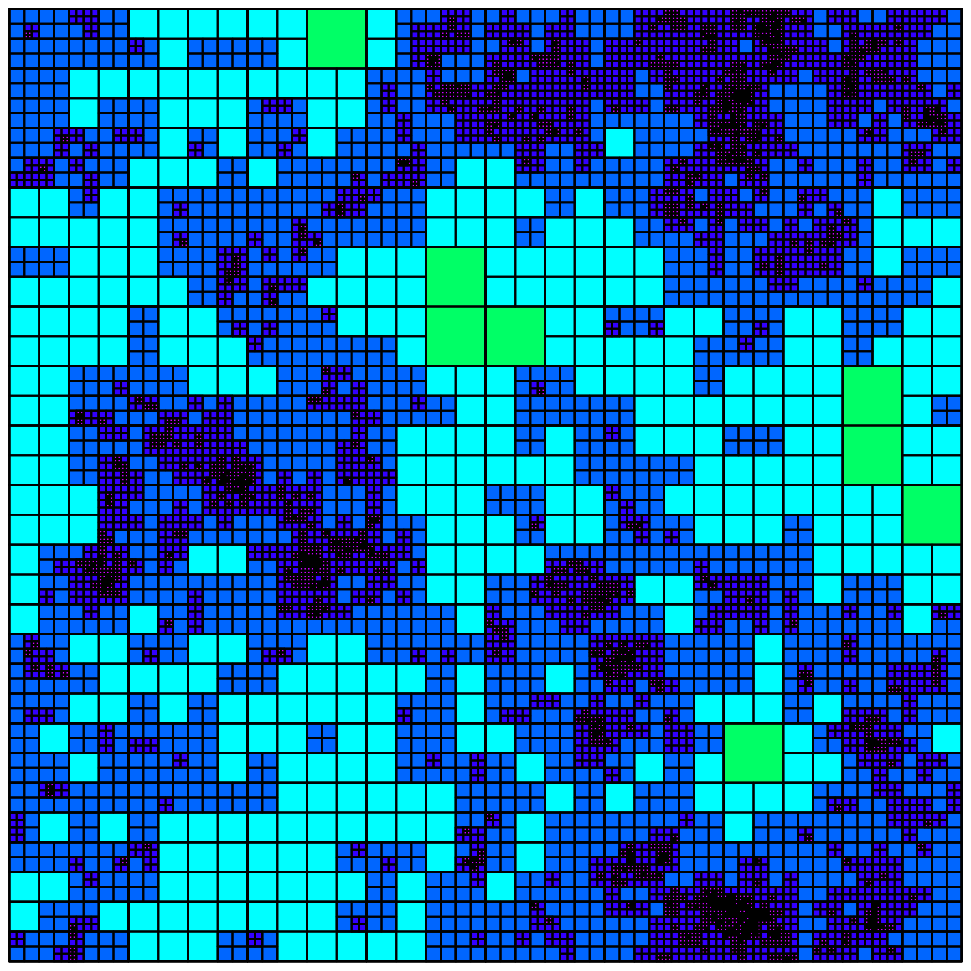}
\end{center}
\caption{\label{f.QG3} Analog of Figure \ref{f.QG1} with $\gamma =
1.5$, using the same instance $h$ of the GFF.}
\end{figure}

When $\varepsilon$ is a power of $2$, analogously define
$S_\varepsilon(z)$ to be the diadic square containing $z$ with edge
length $\varepsilon$. Likewise, define
$$S_\varepsilon(X) = \{z:S_\varepsilon(z) \cap X \not = \emptyset\},$$
$$S^\delta(X) = \{z : S^\delta(z) \cap X \not = \emptyset \}.$$  The
following gives the equivalence of the scaling dimension definition
when boxes are used instead of balls.  (The first half is well known
and easy to verify.)

\begin{proposition} \label{p.boxscalingdimension}
Fix $\gamma \in [0,2)$ and let $X$ be a random subset of a
deterministic compact subset $\tilde D$ of $D$. Let $N(\mu, \delta, X)$ be the
number of $(\mu, \delta)$ boxes intersected by $X$ and
$N(\varepsilon, X)$ the number of diadic squares intersecting $X$
that have edge length $\varepsilon$ (a power of $2$).  Then $X$ has
Euclidean scaling exponent $x \geq 0$ if and only if
$$\lim_{\varepsilon \to 0} \frac{ \log \mathbb E [\mu_0(S_\varepsilon(X))] }{\log \varepsilon^2} =
\lim_{\varepsilon \to 0} \frac{ \log \mathbb E [\varepsilon^2 N(\varepsilon, X)] }{\log \varepsilon^2} = x,$$
or equivalently,
$$\lim_{\varepsilon \to 0} \frac{ \log \mathbb E [N(\varepsilon, X)] }{\log \varepsilon^2} = x-1.$$

Similarly, the following are equivalent
\begin{enumerate}
\item $X$ has {\bf quantum scaling exponent $\Delta$}.
\item When $X$ and $\mu$ (as defined above) are chosen
independently we have
\begin{equation} \label{e.qse} \lim_{\delta
\to 0} \frac{ \log \mathbb E[\mu(S^\delta(X))] }{\log
\delta} = \Delta.\end{equation}
\item When $X$ and $\mu$ (as defined above) are chosen
independently we have
\begin{equation} \label{e.ne} \lim_{\delta \to 0} \frac{ \log \mathbb E[N(\mu, \delta, X)] }{\log \delta} = \Delta-1.\end{equation}
\end{enumerate}
\end{proposition}

Of course, this immediately implies the following restatement of
Theorem \ref{t.strongquantumKPZ} in terms of boxes instead of balls:

\begin{corollary} \label{c.strongquantumKPZ2}
Fix $\gamma \in [0,2)$ and a compact subset $\tilde D$ of $D$ and
$X$ and $\mu$ as above. Then if
$$\lim_{\varepsilon \to 0} \frac{ \log \mathbb E [N(\varepsilon, X)] }{\log \varepsilon^2} = x-1.$$
for some $x > 0$ then
$$\lim_{\delta \to 0} \frac{ \log \mathbb E[N(\mu, \delta, X)] }{\log \delta} = \Delta-1,$$
where $\Delta$ is the non-negative solution to (\ref{KPZ3}).
\end{corollary}

One could also phrase Theorem \ref{t.verystrongquantumKPZ} in terms
of boxes instead of balls, but for simplicity we will refrain from
doing this here.

\section{Coordinate changes and the physical Liouville action}
\label{s.coordchange}

Polyakov understood early on that the Liouville quantum gravity
action becomes a free field action in the conformal gauge, but he
did not construct the random area measure the way we do. In
\cite{Polyakov:1987zb}, where Polyakov begins the KPZ derivation, he
refers to the Liouville quantum gravity action and writes
\begin{quotation} {\it ``The most simple form this formula takes is in the conformal
gauge, where $g_{ab} = e^\varphi \delta_{ab}$ where it becomes a
free field action.  Unfortunately this simplicity is an illusion.
We have to set a cut-off in quantizing this theory, such that it is
compatible with general covariance.  Generally, it is not clear how
to do this. For that reason, we take a different approach.''}
\end{quotation}
Indeed, the actual derivation given in \cite{Polyakov:1987zb} and
subsequently in Knizhnik, Polyakov, and Zamolodchikov
\cite{MR947880} is more complicated than ours and is not based on
the Gaussian free field. It does not give precise mathematical
meaning to the random surfaces. We feel that the Gaussian free field
based random measure we construct is the correct one, at least in the
sense that it is likely to arise as a scaling limit of the discrete
quantum gravity models mentioned in \cite{MR947880} (see Section
\ref{discreteheuristicoverview}).  In a way our approach is more
similar to the work of David \cite{MR981529} and of Distler and
Kawai \cite{MR1005268}, which heuristically derived KPZ from
Liouville field theory in the so-called conformal gauge.

In this section, we describe how the Liouville quantum gravity
measure we construct transforms covariantly under coordinate changes
and use this to explain the connection between the Gaussian free
field and the more familiar and more general curvature-based
definition of the Liouville action that is conventional in the
physics literature. The covariance properties of the random measures
in our point of view are very simple and agree with those postulated
in the physics literature.

If $\phi$ is a conformal map from $D$ to a domain $\tilde D$ and $h$
is a distribution on $D$, then we define the pullback $h \circ
\phi^{-1}$ of $h$ to be a distribution on $\tilde D$ defined by $(h
\circ \phi^{-1}, \tilde \density) = (h, \density)$ whenever
$\density$ is smooth and compactly supported on $D$ and $\tilde \density = |\phi'|^{-2} \density
\circ \phi^{-1}$.  (Here $\phi'$ is the complex derivative of
$\phi$, and $(h,\density)$ is the value of the distribution $h$
integrated against $\density$.) Note that if $h$ is a continuous
function (viewed as a distribution via the map $\density \to
\int_{D} \density(z) h(z) dz$), then the distribution $h \circ
\phi^{-1}$ thus defined is the ordinary composition of $h$ and
$\phi^{-1}$ (viewed as a distribution).

The following transformation rule is a simple consequence of
Proposition \ref{p.hnlimit} and the definitions above.

\begin{proposition} \label{Qtransformation}
Let $h$ be an instance of the GFF on $D$ and $\psi$ a conformal map
from a domain $\tilde D$ to $D$.  Write $\tilde h$ for the
distribution on $\tilde D$ given by $h \circ \psi + Q \log |\psi'|$
where $$Q = \frac{2}{\gamma} + \frac{\gamma}{2}.$$ Then $\mu_h$ is
almost surely the image under $\psi$ of the measure $\mu_{\tilde h}$
on $\tilde D$. That is, $\mu_{\tilde h}(A) = \mu_h(\psi(A))$ for
each Borel measurable $A \subset \tilde D$.
\end{proposition}

\proof Using the notation of Proposition \ref{p.hnlimit}, if $f_1,
f_2, \ldots$ are an orthonormal basis for $H(D)$, then the conformal
invariance of $(\cdot, \cdot)_\nabla$ implies that $f_1 \circ \psi,
f_2 \circ \psi, \ldots$ are an orthonormal basis for $H(\tilde D)$,
and as $n \to \infty$ the functions $h^n \circ \psi$ converge in law
to the GFF on $\tilde D$, and the functions $\tilde h^n = h^n \circ
\psi + Q \log |\psi'|$ converge in law to $\tilde h$. If we define
$\tilde \mu^n$ analogously to $\mu^n$ in (\ref{e.hn}) but with $h^n$
replaced by $\tilde h^n$, then the $\tilde \mu^n$ converge weakly to
the random distribution $\tilde \mu := \mu_{\tilde h}$.

To see that $\mu$ is the image of $\tilde \mu$ under $\psi$, we will
observe that $\mu^n$ is the image of $\tilde \mu^n$ under $\psi$ for
each $n$.  To see this, consider the term $Q \log |\psi'| =
(2/\gamma) \log |\psi'| + (\gamma/2) \log |\psi'|$ in the definition
of $\tilde h$.  Adding $(2/\gamma) \log |\psi'|$ to $h^n \circ \psi$
corresponds to multiplying (\ref{e.hn}) by a factor of $|\psi'|^2$.
This compensates for the fact that the Radon-Nikodym derivative of a
measure on $\tilde D$ at a point $z$ and the Radon-Nikodym
derivative of the same measure pushed forward on $D$ at $\psi(z)$
differ by a factor of $|\psi'(z)|^2$. Adding $(\gamma/2)\log
|\psi'(z)|$ to $h^n \circ \psi$ compensates the expression
(\ref{e.hn}) for the change in conformal radius: $\log C(\psi(z); D)
- \log C(z;\tilde D) = \log |\psi'(z)|$. \qed

We interpret Proposition \ref{Qtransformation} as a rule for
changing the parametrization of a ``random surface.''  For example,
consider the random surface one constructs from the Gaussian free
field on a fixed domain $D$. Then if we are given any other domain $\tilde D$ and a
conformal map $\psi: \tilde D \to D$, we may wish to consider the
same random surface parameterized by $\tilde D$ instead of $D$.  In
this case, the transformation rule tells us that on $\tilde D$ we
should consider the Liouville quantum gravity measure defined using
$\tilde h=h\circ \psi + Q \log |\psi'|$, where $h\circ \psi$ is the
GFF on $\tilde D$ with zero boundary conditions.

We remark that one can make a similar argument when $\tilde D$ is a
curved simply connected manifold and $\psi: \tilde D \to D$ a
conformal map; the metric on $\tilde D$
--- when parameterized by $D$ using the map $\psi^{-1}$ --- takes
the form $e^{\lambda(z)} (d x^2 + d y^2)$, for $z \in D$, where we write
$ \lambda (\psi(w)) = -2\log |\psi'(w)|$ for $w \in \tilde D$.
Although we will not prove it here, the analog of Proposition
\ref{p.hepsilonlimit} for smooth curved surfaces is straightforward,
and the transformation rule Proposition \ref{Qtransformation}
remains the same in this case; as in the flat case, the law of the
Liouville quantum gravity measure on $D$ pulled back to $\tilde D$
is that of $\tilde h= h\circ \psi + Q \log |\psi'|$ where $h\circ
\psi$ is the GFF on $\tilde D$ with zero boundary conditions.
(Alternatively, we may take this as a definition of the Liouville
quantum gravity measure on curved $\tilde D$ with zero boundary
conditions.)

The remainder of this subsection describes the connection between
our notation and the common physics literature Liouville gravity
notation.  (This discussion can be skipped, on a first read, by
readers with no prior familiarity with the latter.) What we call the
GFF on $D$ (with the $1/2\pi$ normalization, as discussed in the
introduction) is often written (sometimes without a rigorous
definition) as the measure $e^{-S(h)}dh$, where
$$S(h) = \frac{1}{4\pi} \int_D \nabla h(z) \cdot \nabla h(z)dz$$ is
called the {\bf action} and $dh$ is defined heuristically as a
``uniform measure on the space of all functions.''  (Of course, the
latter makes perfect sense if one considers only a finite
dimensional vector space of functions, such as real-valued functions
defined on the vertices of a lattice, or functions whose Fourier
coefficients beyond a certain frequency threshold are identically
zero---in this case $dh$ would be the Lebesgue measure on the vector
space.) In this paper, we will write
$$(h,h)_\nabla := \frac{1}{2\pi} \int_D \nabla h(z) \cdot \nabla
h(z)dz,$$ so that the above becomes $S(h) = \frac12 (h, h)_\nabla$.

In the following, let $D$ be a subdomain of $\C$ and $\tilde D$ a
possibly curved surface for which there is a conformal map $\psi:
\tilde D \to D$.  Write $\tilde h^0 = \log |\psi'|$. Now, if we
switch parametrization to $\tilde D$, we are adding $Q \tilde h^0$
deterministically to $h\circ \psi$ to get $\tilde h$, so we may
rewrite the action as
$$S = \frac12 (\tilde h - Q \tilde h^0, \tilde h - Q \tilde h^0)_\nabla,
$$
which (at least when $\tilde h^0$ is smooth and compactly supported)
is seen by integrating by parts to be equivalent (up to the additive
constant $\frac12 \|Q\tilde h^0\|_\nabla^2$) to
\begin{equation}\label{e.transformedaction} S = \frac1{4\pi} \int_{\tilde D} dw \left( \nabla
\tilde h(w) \cdot \nabla \tilde h(w) + 2\tilde h(w) Q \Delta \tilde
h^0(w)\right),
\end{equation} where the pairing $\nabla \tilde h(w) \cdot \nabla \tilde h(w)$ and the Laplacian
$\Delta \tilde h^0(w)$ are now defined using the metric on $\tilde
D$ and where now $dw$ represents the measure on $\tilde D$ instead
of $D$.  This can also be written
\begin{equation}\label{e.transformedaction2} S = \frac1{4\pi} \int_{\tilde D} dw \left( \nabla
\tilde h(w) \cdot \nabla \tilde h(w) + Q \tilde h(w) K(w) \right),
\end{equation}
where $K$ is the Gaussian curvature of $\tilde D$ and $dw$ is
integration with respect to the curved metric.  (When $\tilde h^0$ is not
compactly supported, the formula can be modified to include a term
for boundary curvature, but we will not discuss this here.)

Adding in one additional term which is a constant $\mu_{\mathcal L}$ times the
total area of $\tilde D$ (and making the following symbol
substitutions: $b = \gamma/2$, $\varphi = \tilde h$, $g$ is the
underlying metric of $\tilde D$, and $j$ and $k$ are summed-over
indices ranging over the two tangent space directions), we obtain
the more familiar formula for the Liouville action:
$$S_{\mathcal L} = \frac{1}{4 \pi} \int_{\tilde D} dw \sqrt{g} \left(g^{jk} \partial_j \varphi \partial_k \varphi
+ Q K \varphi + 4 \pi \mu_{\mathcal L} e^{2 b \varphi} \right),$$ where $Q = b +
b^{-1}$.  The action is defined similarly when free boundary
conditions are used instead of zero boundary conditions --- or when
$\tilde D$ is a compact Riemann surface of some genus.  (In this
case, $e^{-S_{\mathcal L}(\varphi)} d\varphi$ is an infinite measure, although it
can be ``localized,'' e.g., by requiring the mean value of $\varphi$
to be zero.)

This paper will focus exclusively on the case $\gamma \in [0,2)$
(which is said to correspond to physical models below the central
charge $c=1$ threshold) and $\mu_{\mathcal L}=0$ (the so-called {\bf critical}
Liouville quantum gravity). The string theory and quantum gravity
literatures deal with other parameter choices as well --- including
non-zero $\mu_{\mathcal L}$ and complex values for $\gamma$ and $Q$ --- but these
appear to be beyond the scope of our methodology, in part because,
when $S_{\mathcal L}$ is complex valued, the expression $e^{-S_{\mathcal L}(\varphi)}d\varphi$
is no longer a probability measure in even a heuristic sense.

\section{Constructing the random measures}

\subsection{GFF definition and normalization} \label{ss.GFFnormsection}
Let $D$ be a bounded planar domain and let $dz$ denote Lebesgue
measure on $D$. We assume the reader is familiar with the Gaussian
free field, as defined, e.g., in \cite{\GFFSurvey}, but we briefly
review the definition here. As described earlier, to make our
formulas consistent with the physics literature, the definitions of
Green's function and the Dirichlet form will differ from the ones in
\cite{\GFFSurvey} by factors of $2\pi$.


Again, let $H_s(D)$ be the space of $C^\infty$ real-valued functions
compactly supported on $D$.  We define the Dirichlet inner product
$$(f_1, f_2)_\nabla := (2\pi)^{-1} \int_D \nabla f_1(z) \cdot \nabla f_2(z) dz,$$
on $H_s(D)$.  Then an instance $h$ of the Gaussian free field (GFF)
may be viewed as a standard Gaussian on the Hilbert space closure
$H(D)$ of $H_s(D)$ (i.e., as a sum of the form $\sum_n \alpha_n f_n$
where $f_n$ are any orthonormal basis for $H(D)$)
--- the sum converges almost surely in the space of distributions on
$D$, see \cite{\GFFSurvey}.  In fact, we may define $(h, f)_\nabla$
as random variables for non-smooth $f$ as well; these are zero mean
Gaussian random variables for each $f \in H(D)$, and
\begin{equation}
\label{Cov}
\Cov
\bigl((h, f_1)_\nabla,(h, f_2)_\nabla \bigr) = (f_1, f_2)_\nabla.
\end{equation}
The collection of random variables $(h, f)_\nabla$ for $f \in H(D)$
is thus a Hilbert space (isomorphic to $H(D)$) under the covariance
inner product.

When $x \in D$ is fixed, we let $\tilde G_x(y)$ be the harmonic
extension to $y \in D$ of the function of $y$ on $\partial D$ given by
$-\log|y-x|$. Then {\bf Green's function in the domain $D$} is
defined by
$$G(x,y) = -\log|y-x| - \tilde G_x(y).$$
When $x \in D$ is fixed, Green's function may be viewed as a
distributional solution of $\Delta G(x,\cdot)= -2 \pi
\delta_x(\cdot)$ with zero boundary conditions \cite{\GFFSurvey}.
It is non-negative for all $x, y \in D$.

For any function $\density$ on $H_s(D)$, we define a function $\Delta^{-1}
\density$ on $D$ by $$
\Delta^{-1}
\density(\cdot) := -\frac{1}{2\pi} \int_D G(\cdot
,y)\,\density (y)\,dy.$$ This is a
$C^\infty$ (though not necessarily compactly supported) function in
$D$ whose Laplacian is $\density$.  We use the same notation for
more general measurable functions $\density$, as well as the case
that $\density$ is a measure.  (For example, we will sometimes speak
of the inverse Laplacian of uniform measure on a particular circle
or disc contained in $D$.)

If $f_1 = - \Delta^{-1} \density_1$ and $f_2 = -\Delta^{-1}
\density_2$, then integration by parts implies that $(f_1,
f_2)_\nabla = (2\pi)^{-1} (\density_1,-\Delta^{-1} \density_2),$
where $(\cdot, \cdot)$ denotes the standard inner product on
$L^2(D)$.  We next observe that every $h \in H(D)$ is naturally a
distribution, since we may define the map $(h, \cdot)$ by $(h,
\density) := 2\pi (h, -\Delta^{-1} \density)_\nabla$.  (It is not
hard to see that $-\Delta^{-1} \density \in H(D)$, since its
Dirichlet energy is given explicitly by \eqref{e.greencovariance}.)
When $-\Delta f = \density$, we may write $(h, \density) = 2\pi (h,
f)_\nabla$, and hence $$\Cov \bigl( (h, \density_1), (h, \density_2)
\bigr) = (2\pi)^2 (f_1, f_2)_\nabla.$$

We claim that the latter expression may be rewritten to give
\begin{equation}\label{e.greencovariance}\Cov \bigl( (h,
\density_1), (h, \density_2) \bigr) = \int_{D \times D}
\density_1(x)  G(x,y) \density_2(y)\,\, dx \, dy\end{equation} where
$G(x,y)$ is Green's function in $D$.  Since $\Delta G(x,\cdot)=-2\pi
\delta_x(\cdot)$ and $$\int_D G(x,y)\,\density_2 (y)\,dy = -2\pi
\Delta^{-1}\density_2(x),$$ we obtain \eqref{e.greencovariance} by
multiplying each side by $-\Delta f_1(x) = \density_1(x)$ and
integrating by parts with respect to $x$.

Denote by $h_\varepsilon(z)$ the average value of $h$ on the circle
of radius $\varepsilon$ centered at $z$.  Similar averages were
considered in \cite{bauer1990}.  (For this definition, we assume $h$
is identically zero outside of $D$.) Then $h_\varepsilon(z)$ is a
Gaussian process with covariances defined by
\begin{equation} \label{e.Gepsilon} G_{\varepsilon_1, \varepsilon_2}(z_1, z_2) := \Cov \left(
h_{\varepsilon_1}(z_1), h_{\varepsilon_2}(z_2) \right)
\end{equation}
given by
$$\int G(x,y) \rho_{\varepsilon_1}^{z_1}(x) \rho_{\varepsilon_2}^{z_2}(y)dxdy$$ where $\rho_{\varepsilon}^z(x)dx$
is the uniform measure (of total mass one) on $\partial
B_\varepsilon(z)$. In fact (like Brownian motion) the process
$h_\varepsilon(z)$ determines a random continuous function (of $z$
and $\varepsilon$):

\begin{proposition} \label{p.holdercontinuous}
The process $h_\varepsilon(z)$ has a modification which is almost
surely locally $\eta$-H\"older continuous in the pair $(z,
\varepsilon) \in \C \times (0, \infty)$ for every $\eta < 1/2$.
\end{proposition}

In other words, the H\"older regularity enjoyed by
$h_\varepsilon(z)$
--- as a function of the pair $(z,\varepsilon)$ --- is the same as that of
Brownian motion or the Brownian sheet. In fact, as we observe below
(Proposition \ref{p.browniancircleaverages}), when $z$ is fixed,
$h_\varepsilon$ {\em is} a Brownian motion with respect to the
parameter $t=-\log \varepsilon$. We may view $h_\varepsilon$ as an
approximation to $h$ that gets better as $\varepsilon \to 0$. Before
we prove Proposition \ref{p.holdercontinuous}, let us make some
observations about the covariance function $G_{\varepsilon_1,
\varepsilon_2}(z_1, z_2)$ defined in (\ref{e.Gepsilon}).  (We will
also sometimes write $G_\varepsilon(z_1, z_2) := G_{\varepsilon,
\varepsilon}(z_1, z_2)$.)

First we define the function $\xi_\varepsilon^z(y)$, for $y \in D$,
by
\begin{equation}
\label{defxi}
\xi_\varepsilon^z(y) = -\log \text{max}(\varepsilon, |z-y|)-\tilde G_{z,\varepsilon}(y),
\end{equation}
where $\tilde G_{z,\varepsilon}(y)$ is the harmonic extension to $D$ of the restriction of
$-\log \text{max}(\varepsilon, |z-y|)$ to $\partial D$. Note that $\tilde G_{z, \varepsilon} = \tilde G_z$ provided that
$B_\varepsilon(z) \subset D$.  Observe that
this $\xi_\varepsilon^z(y)$ tends to zero as $y \to
\partial D$ and that as a distribution $-\Delta \xi_\varepsilon^z$
(restricted to $D$) is equal to $2\pi \rho_\varepsilon^z$, where as
before $\rho_\varepsilon^z$ is a uniform measure on $D \cap
\partial B_\varepsilon(z)$.
Integrating by parts, we immediately have
$$
h_\varepsilon(z)=(h,\xi_\varepsilon^z)_\nabla ,
$$
and from \eref{Cov} and (\ref{e.Gepsilon}) the following:
\begin{proposition} \label{p.hepsiloncovariance}
The function $G_{\varepsilon_1, \varepsilon_2}(z_1, z_2)$ is equal to the Dirichlet inner product
$\left(\xi_{\varepsilon_1}^{z_1},\xi_{\varepsilon_2}^{z_2}\right)_\nabla$ and
to the mean value of $\xi_{\varepsilon_1}^{z_1}$ on the circle
$\partial B_{\varepsilon_2}(z_2)$.  In particular, if
$B_{\varepsilon_1}(z_1)$ and $B_{\varepsilon_2}(z_2)$ are disjoint
and contained in $D$ then $G_{\varepsilon_1, \varepsilon_2}(z_1,
z_2) = G(z_1, z_2)$.  If $B_{\varepsilon_1}(z) \subset D$ and
$\varepsilon_1 \geq \varepsilon_2$ then
$$G_{\varepsilon_1, \varepsilon_2}(z, z) = -\log \varepsilon_1 + \log
C(z; D).$$
\end{proposition}
It then follows that
\begin{equation}
\label{xixi}
G_{\varepsilon,\varepsilon}(z,z)=\Var h_\varepsilon(z)=(\xi_\varepsilon^z,\xi_\varepsilon^z)_\nabla=
\xi_\varepsilon^z(z)=-\log \varepsilon +\log C(z;D).
\end{equation}
\proofof{Proposition \ref{p.holdercontinuous}} We first claim that
for each $\varepsilon_0$ and $D$ there exists a constant $K$ such
that
$$\Var \left( h_{\varepsilon_1}(z_1) - h_{\varepsilon_2}(z_2) \right) \leq
K\bigl[|z_1 - z_2| + |\varepsilon_1 - \varepsilon_2|\bigr]$$ for all $z_1, z_2
\in D$ and $\varepsilon_1, \varepsilon_2 \in [\varepsilon_0,
\infty)$. Since the variance can only increase if $D$ is replaced
with a larger domain, it suffices to show this holds when $D$ is
replaced by a sufficiently large disc $D'$ (say, centered in $D$
with $10$ times the diameter $r$ of $D$), and $\varepsilon$ is
restricted to values in $[\varepsilon_0, 5r]$.  (For larger values
of $\varepsilon$, the set $\partial B_\varepsilon(z)$ cannot
intersect $D$ when $z \in D$.) Since
$$\Var \left( h_{\varepsilon_1}(z_1) - h_{\varepsilon_2}(z_2) \right) =
G_{\varepsilon_1, \varepsilon_1}(z_1, z_1) - 2G_{\varepsilon_1,
\varepsilon_2}(z_1, z_2) + G_{\varepsilon_2, \varepsilon_2}(z_2,
z_2),$$it suffices to show that $G_{\varepsilon_1,
\varepsilon_2}(z_1, z_2)$ is a Lipschitz function of
$(\varepsilon_1, \varepsilon_2, z_1, z_2)$ for the range of
$(\varepsilon_1, \varepsilon_2, z_1, z_2)$ values indicated above.
This follows from Proposition \ref{p.hepsiloncovariance} and the
fact (whose proof we leave to the reader) that $\xi_\varepsilon^z$
is a Lipschitz function when $z\in D$ and $\varepsilon
> \varepsilon_0$, with a Lipschitz constant that holds uniformly
over these $\varepsilon$ and $z$ values.

The claim implies that for all $\alpha>0$ there is some $K = K(\alpha)>0$ such that
$$\mathbb E[ |h_{\varepsilon_1}(z_1) - h_{\varepsilon_2}(z_2)|^\alpha] \leq
K\bigl[|z_1-z_2| + |\varepsilon_1 - \varepsilon_2| \bigr]^{\alpha/2}.$$  This puts us in the
setting of the multiparameter Kolmogorov-\v{C}entsov theorem
\cite{MR1121940, MortersPeres}, which states the following: Suppose
that the random field $X(a)$, $a \in \prod_{i=1}^n[0,t_i]$ satisfies
$\mathbb E[|X(a) - X(b)|^\alpha] \leq K|a-b|^{n+\beta}$ for all
$a,b$, for some fixed constants $\alpha, \beta, K$.  Then there
exists an almost surely continuous modification of the random field
and this process is $\eta$-H\"older continuous for every $\eta <
\beta/\alpha$. Applying this for $n=3$ and $\beta = \alpha/2 - 3$ and large $\alpha$
allows us to deduce that $h_\varepsilon(z)$, as a function of $\varepsilon$ and $z$, is
locally $\eta$-H\"older continuous for all $\eta < 1/2$.  \qed

\begin{proposition} \label{p.browniancircleaverages} Write ${\mathcal V}_t = h_{e^{-t}}(z)$, and $t_0^z = \inf \{t:
B_{e^{-t}}(z) \subset D \}$.  If $z\in D$ is fixed, then the law of
$$V_t:={\mathcal V}_{t_0^z + t} - {\mathcal V}_{t_0^z}$$ is a standard
Brownian motion in $t$.
\end{proposition}

\proof Since we already know that the $h_\varepsilon(z)$ are jointly
Gaussian random variables, it is enough to compute the variances of
$h_\varepsilon(z)$ and $h_{\varepsilon'}(z)$ for fixed $\varepsilon,
\varepsilon'$, and these are given in Proposition
\ref{p.hepsiloncovariance}. \qed

\subsection{Random measure: Liouville quantum gravity}
\label{ss.randommetrics}


The remainder of the paper makes frequent use of the following
simple fact, which the reader may recall (or verify): if $N$ is a
Gaussian random variable with mean $a$ and variance $b$ then
\begin{equation}\label{e.expnormal}\mathbb E \, e^N =
e^{a+b/2}.\end{equation} Since $\mathbb E h_\varepsilon(z) = 0$ when
$h$ is an instance of the GFF with zero boundary conditions, we have
$$\mathbb E e^{\gamma h_\varepsilon(z)} = e^{\text{Var}[ \gamma
h_\varepsilon(z)]/2}.$$ Recall \begin{equation} \label{bertrandequation} \Var(h_\varepsilon(z)) = G_\varepsilon(z,z) = \log C(z;D) -
\log \varepsilon\end{equation} when $B_\varepsilon(z) \subset D$. Then we have
\begin{equation} \label{e.expexphepsilon} \mathbb E e^{\gamma h_\varepsilon}(z) = \exp
\left(\gamma^2/2 \left(-\log \varepsilon + \log C(z; D)\right)
\right) = \left( \frac{C(z;D)}{\varepsilon} \right)^{\gamma^2/2}.
\end{equation}

More general moments of the random variables $e^{\gamma
h_\varepsilon(z)}$ are also easy to calculate.  For example, we have
\begin{eqnarray}\label{e.exphepsilonmoments} \nonumber \mathbb E e^{\gamma h_\varepsilon(y)}e^{\gamma h_\varepsilon(z)} &=&
\exp \bigl( \text{Var} [\gamma (h_\varepsilon(y) +
h_\varepsilon(z))]/2 \bigr) \\ &=& \exp \left(\frac{\gamma^2}{2}\left(
G_\varepsilon(y,y) + G_\varepsilon(z,z) +
2G_\varepsilon(y,z)\right)\right).\end{eqnarray}

By Proposition \ref{p.hepsiloncovariance} we have
$G_\varepsilon(y,z) = G(y,z)$ whenever $|y-z| \geq 2 \varepsilon$
and $B_\varepsilon(y) \cup B_\varepsilon(z) \subset D$. In this case
we have

$$\mathbb E e^{\gamma h_\varepsilon(y)}e^{\gamma h_\varepsilon(z)} =
\left(\frac{C(y;D)C(z;D)}{\varepsilon^2}\right)^{\gamma^2/2}
e^{\gamma^2 G(y,z)}.$$

Write $$\overline h_\varepsilon := \gamma h_\varepsilon +
\frac{\gamma^2}{2} \log \varepsilon.$$ Then we have
$$\mathbb E e^{\overline h_\varepsilon(z)} = C(z;D)^{\gamma^2/2}
\asymp 1$$ and when $|y-z| > 2\varepsilon$ we have
$$\mathbb E e^{\overline h_\varepsilon(y)}e^{\overline h_\varepsilon(z)} =
\left(C(y;D)\,C(z;D)\right)^{\gamma^2/2} e^{\gamma^2 G(y,z)} \asymp
\left(C(y;D)\, C(z;D)\right)^{\gamma^2/2} |y-z|^{-\gamma^2} \asymp
|y-z|^{-\gamma^2}$$ where $\asymp$ indicates that equality holds up
to a constant factor when $y$ and $z$ are restricted to any compact
subset of $D$.

Now, for each fixed $\varepsilon$, write $\mu_\varepsilon :=
e^{\overline h_\varepsilon(z)} dz$ (which in essence corresponds to
the ``Wick normal ordering'' of the original measure \cite{BSimon}).
We now argue that these converge weakly to a limiting random measure
on $D$.

\proofof{Proposition \ref{p.hepsilonlimit}}  Fix $\gamma \in [0,2)$.
It is easy to see that if for each diadic square $S$ compactly
supported in $D$ the random variables $\mu_{2^{-k}}(S)$ converge to
a finite limit as $k \to \infty$, almost surely, then $\mu_{2^{-k}}$
almost surely converges weakly to a limiting measure.
We will prove
convergence of $\mu_{2^{-k}}(S)$ by showing that the expectation of
$|\mu_{2^{-k}}(S) - \mu_{2^{-k-1}}(S)|$ decays exponentially in $k$.
 (Convergence follows, e.g., by using the Borel-Cantelli lemma to show
that a.s.\ $|\mu_{2^{-k}}(S) - \mu_{2^{-k-1}}(S)|$ is greater than some exponentially
decaying function of $k$ for at most finitely many $k$.)
Without loss of generality, we may assume $S$ is the unit square
$[0,1]^2$, so that $\mu_\varepsilon(S)$ is precisely the mean value
of $e^{\overline h_\varepsilon(z)}$ on $S$.

As shown above, we have $$\mathbb E e^{\overline h_\varepsilon(z)} =
C(z;D)^{\gamma^2/2},$$ (which is bounded between positive constants)
when $z \in S$ and $\varepsilon$ is sufficiently small.

For $y=(y_1, y_2) \in (0,1)^2$ and $k \geq 1$, let $S_k^y$ be the discrete
set of $2^{2k}$ points $(a,b) \in S$ with the property that $(2^k a
-2^k y_1, 2^k b - 2^k y_2) \in \Z^2$.
 Let $A_k^y$ be the mean value of $\exp{\overline h_{2^{-k-1}}(z)}$ on
the set $S_k^y$, and  $B_k^y$  the mean value of $\exp{\overline
h_{2^{-k-2}}(z)}$ over the same set:
\begin{eqnarray}\label{Aky}
A_k^y:=2^{-2k}\sum_{z \in S_k^y}\exp{\overline h_{2^{-k-1}}(z)},\,\,\,\,
B_k^y:=2^{-2k}\sum_{z \in S_k^y}\exp{\overline h_{2^{-k-2}}(z)}.
\end{eqnarray}
Clearly, $\mu_{2^{-k-1}}(S)$ is the mean value of $A_k^y$ over $y
\in [0,1]^2$ and $\mu_{2^{-k-2}}(S)$ the mean value of $B_k^y$ over
$y \in [0,1]^2$.  Applying Jensen's inequality to the convex
function $|\cdot|$, it now suffices for us to prove that $\mathbb E
|A_k^y - B_k^y|$ tends to zero exponentially in $k$ (uniformly in
$y$). Since the balls of radius $2^{-k-1}$ centered at points in
$S_k^y$ do not overlap, and by the {\bf Markov property} of the GFF
(see, e.g., \cite{MR2322706}),
 we have that conditioned on the values of
$h_{2^{-k-1}}(z)$ for $z \in S_k^y$, the random variables
$h_{2^{-k-2}}(z)$ for $z \in S_k^y$ are independent of one another;
thanks to Propositions \ref {p.hepsiloncovariance} and \ref{p.browniancircleaverages},
each is a Gaussian of variance $\log 2$ and mean $h_{2^{-k-1}}(z)$.

Hence, given the values of $h_{2^{-k-1}}(z)$ for $z \in S_k^y$, the
value of the conditional expectation of $|A_k^y - B_k^y|^2$ is given by
\begin{eqnarray}\nonumber
\mathbb E \left(|A_k^y - B_k^y|^2|h_{2^{-k-1}}(z)\right)&=&
2^{-4k} \sum_{z \in S_k^y} \mathbb E \left(|
e^{\overline h_{2^{-k-1}}(z)} - e^{\overline h_{2^{-k-2}}(z)}|^2 |
h_{2^{-k-1}}(z) \right)\\\label{e.varsum}
 &=& 2^{-4k} C \sum_{z \in S_k^y} \left(
e^{\overline h_{2^{-k-1}}(z)} \right)^2,
\end{eqnarray}
where
$$
C=\mathbb E \left( |1 - {2}^{-\gamma^2/2}e^{\gamma
h_{2^{-k-2}}(z)}|^2 | h_{2^{-k-1}}(z) = 0 \right)=2^{\gamma^2}-1,
$$
 is a constant independent on $k$ and $z$.  The unconditional
expectation of $|A_k^y - B_k^y|^2$ is given by the expectation of
(\ref{e.varsum}).
It is tempting to argue that this expectation
tends to zero exponentially in $k$ (which would in turn imply that
$\mathbb E |A_k^y - B_k^y|$ tends to zero exponentially in $k$), but
this turns out to be true only for $0 \leq \gamma^2 < 2$ and not for
$2 \leq \gamma^2 < 4$. To see this, set $\varepsilon:=2^{-k-1}$, and note that
\begin{equation} \label{e.evar} \mathbb E \left[ (\varepsilon^2 e^{\overline
h_\varepsilon(z)})^2\right] = \varepsilon^{4+\gamma^2} \mathbb E [e^{2 \gamma
h_\varepsilon}] \asymp \varepsilon^{4+\gamma^2}
e^{- \frac{4 \gamma^2 \log \varepsilon}{2}} = \varepsilon^{4-\gamma^2}.
\end{equation} Summing over the $2^{2k}=(2\varepsilon)^{-2}$ points $z$ in
$S_k^y$ in (\ref{e.varsum}), yields for
the expectation of the latter (up to
an $\varepsilon$-independent constant factor)
$$
\mathbb E \,|A_k^y - B_k^y|^2 \asymp \varepsilon^{2 -
\gamma^2},
$$
which does not tend to zero when $\gamma^2 \geq 2$.

However, we can deal with the case $\gamma^2 \geq 2$ by breaking the
sum over $z\in S_k^y$ in (\ref{Aky}) defining $A_k^y - B_k^y$ into two parts and dealing with them
separately. The idea is that the estimate in (\ref{e.varsum}) and (\ref{e.evar}) of
the  expectation of $|A_k^y - B_k^y|^2$ is dominated (and can only be made
large) by rare occurrences at points $z$ where $h_\varepsilon(z)$ is much
larger than typical. Their contribution to the expectations of
$A_k^y$ and $B_k^y$, hence to $\mathbb E |A_k^y-B_k^y|$, is exponentially small in $k$.

To make this precise, fix some $\alpha$ with $\gamma < \alpha <
2\gamma$ and let $\tilde S_k^y$ denote set of points $z \in S_k^y$
with the property that $h_\varepsilon(z) > -\alpha \log[\varepsilon/ C(z;D)]$, where  $\varepsilon = 2^{-k-1}$. Let $\tilde A_k^y$
denote the mean value of $1_{\tilde S_k^y} \exp{\overline
h_{2^{-k-1}}(z)}$ over $S_k^y$ and $\tilde B_k^y$ the mean value of $1_{\tilde
S_k^y} \exp{\overline h_{2^{-k-2}}(z)}$  over $S_k^y$, so that
\begin{eqnarray}\label{tildeAky}
A_k^y=\tilde A_k^y +2^{-2k}\sum_{z \in S_k^y\setminus \tilde S_k^y}\exp{\overline h_{2^{-k-1}}(z)},\,\,\,\,
B_k^y=\tilde B_k^y +2^{-2k}\sum_{z \in S_k^y\setminus \tilde S_k^y}\exp{\overline h_{2^{-k-2}}(z)}.
\end{eqnarray}
We claim that $\mathbb E \tilde A_k^y$ tends to zero exponentially
 in $k$. To see this, observe that for fixed $z \in S$, the random
variable $h_\varepsilon(z)$ is a centered Gaussian with variance
$\sigma^2 = - \log [\varepsilon/ C(z;D)];$
 and the expectation
of $e^{\overline h_\varepsilon(z)}$---which we know to be constant
for all $\varepsilon$ small enough so that $z$ is distance at least
$\varepsilon$ from the boundary of $D$---takes the form
\begin{eqnarray}\label{E0}
\mathbb E e^{\overline h_\varepsilon(z)}=\frac{\varepsilon^{\gamma^2/2}}{(2\pi)^{1/2}}\int_{\mathbb R} e^{-\frac{\eta^2}{2\sigma^2}} e^{\gamma \eta}
d\eta=C(z;D)^{\gamma^2/2}.
\end{eqnarray}
We can therefore simply write for points $z\in\tilde S_k^y$:
\begin{eqnarray}\label{E1}\mathbb E \,1_{\tilde S_k^y}e^{\overline h_\varepsilon(z)}=\frac{\int_{\mathbb R} 1_{\eta > \alpha\sigma^2} e^{-\frac{\eta^2}{2\sigma^2}} e^{\gamma\eta} d\eta}
{ \int_{\mathbb R} e^{-\frac{\eta^2}{2\sigma^2}} e^{\gamma \eta} d\eta} \times \mathbb E\, e^{\overline h_\varepsilon(z)}.
\end{eqnarray}
The ratio of integrals in (\ref{E1}) is the probability that a normal random
variable $\eta$ of mean $\gamma \sigma^2$ and variance $\sigma^2$
satisfies $\eta > \alpha \sigma^2$, with $\alpha >\gamma$, and this clearly tends to zero
exponentially in $\sigma^2$ with rate $\frac{1}{2} (\alpha-\gamma)^2$, from which the claim easily follows.

Note that (\ref{E0}) remains unchanged if we replace $\varepsilon = 2^{-k-1}$ with $2^{-k-2}$, which implies that
$\mathbb E
\tilde B_k^y = \mathbb E \tilde A_k^y$, and in particular $\mathbb E
\tilde B_k^y$ also tends to zero exponentially in $k$. Since $\tilde
B_k^y$ and $\tilde A_k^y$ are non-negative, applying the triangle
inequality shows that $\mathbb E |\tilde B_k^y - \tilde
A_k^y|$ tends to zero exponentially in $k$.

For the next step, we wish to bound $\mathbb E |(B_k^y - \tilde
B_k^y) - (A_k^y - \tilde A_k^y)|^2$, which requires us to consider the
 set of points $z\in S_k^y \setminus \tilde S_k^y$ in (\ref{tildeAky}). Applying formula (\ref{e.varsum}), we are
led to estimate the expectation
\begin{equation} \label{e.evar2} \varepsilon^{4} \mathbb E \left[ \,1_{S_k^y\setminus \tilde S_k^y}
\left(
e^{\overline h_\varepsilon(z)}\right)^2 \right] = \varepsilon^{4+\gamma^2}
\mathbb E \left[ \,1_{S_k^y\setminus \tilde S_k^y} e^{2 \gamma h_\varepsilon(z)}\right].
\end{equation}
Using the explicit relation (compare to (\ref{E1}))
\begin{eqnarray}\label{E2}\mathbb E \left[ \,1_{S_k^y\setminus \tilde S_k^y} e^{2 \gamma h_\varepsilon(z)}\right]
=\frac{\int_{\mathbb R} 1_{\eta < \alpha\sigma^2} e^{-\frac{\eta^2}{2\sigma^2}} e^{2\gamma\eta} d\eta}
{ \int_{\mathbb R} e^{-\frac{\eta^2}{2\sigma^2}} e^{2\gamma \eta} d\eta} \times \mathbb E\, e^{2\gamma h_\varepsilon(z)},
\end{eqnarray}
we find that  (\ref{e.evar2}) differs from
(\ref{e.evar}) by a factor that represents the probability that a
Gaussian random variable with variance $- \log \varepsilon$ (plus a constant term) and mean
$-2 \gamma \log \varepsilon$ is less than $- \alpha \log
\varepsilon$. Since $\alpha < 2 \gamma$, this probability decays
exponentially in $-\log \varepsilon$ at rate $(2\gamma -
\alpha)^2/2$. Thus (\ref{e.evar2}) becomes, up to a constant factor
(universal in $\varepsilon$ and $z \in S$),
$$\varepsilon^{4-\gamma^2} \varepsilon^{(2 \gamma - \alpha)^2/2}.$$
Summing over the $2^{2k}=(2\varepsilon)^{-2}$ points $z\in S_k^y$, we obtain the estimate
$$\mathbb E |(B_k^y - \tilde
B_k^y) - (A_k^y - \tilde A_k^y)|^2 \asymp \varepsilon^{2 - \gamma^2 + (2\gamma - \alpha)^2/2}.
$$
To conclude, we only need to make sure we chose $\alpha \in (\gamma, 2
\gamma)$ small enough so that the sum in the exponent is positive,
and this is clearly possible. In fact, taking $\alpha$ close to
$\gamma$, the exponent becomes close to $2 - \gamma^2 + \gamma^2/2 =
2 - \frac{\gamma^2}{2}$, which is positive when $\gamma < 2$. \qed

\subsection{Rooted random measures}
\label{s.rootedmetrics}

Before proving Proposition \ref{p.hnlimit}, we introduce a notion of
rooted random measure and use it to prove a uniform integrability
result for the random variables $\mu_\varepsilon(S)$ discussed
above.

Write $\Theta_\varepsilon := Z_\varepsilon^{-1} e^{\gamma
h_\varepsilon(z)} d z dh$, where $Z_\varepsilon$ is a constant
chosen to make $\Theta_\varepsilon$ a probability measure.
 Note that $dzdh$ is a probability measure on a standard Borel space and
that $Z_\varepsilon^{-1} e^{\gamma h_\varepsilon(z)}$ is a measurable function
on that space with expectation one.  There is therefore no difficulty in defining
the $\Theta_\varepsilon$ as the measure whose Radon-Nikodym derivative
with respect to $dzdh$ is  $Z_\varepsilon^{-1} e^{\gamma h_\varepsilon(z)}$.
Integrating, we see that the $\Theta_\varepsilon$ marginal distribution of $z$ is
given by $f(z)dz$ where $f(z) = Z_\varepsilon^{-1} \mathbb E_h e^{\gamma
h_\varepsilon(z)}$. Thus $f(z)$ is proportional to $[C(z;D)]^{\gamma^2/2}$ by \eqref{e.expexphepsilon}
(provided $B_\varepsilon(z) \subset D$).  Similarly, the $\Theta_\varepsilon$ marginal law
of $h$ is $Z_\varepsilon^{-1}\left(\int_D e^{\gamma h_\varepsilon(z)} dz\right)dh$.
Given $z$, a regular conditional probability distribution for $h$ is given by
$\bigl( \mathbb E_h e^{\gamma h_\varepsilon(z)}\bigr)^{-1} e^{\gamma h_\varepsilon(z)}dh$.

  In other words,
sampling from $\Theta_\varepsilon$ may be described as a two step
procedure. First sample $z$ from its marginal distribution.  Then
sample $h$ from the distribution of the Gaussian free field {\em
weighted} by $e^{\gamma h_\varepsilon(z)}$.  Since $dh$ is a Gaussian measure,
we find that given $z$, $h$ has the law
of the original GFF {\em plus} $\gamma \xi_\varepsilon^z$ where
$\xi_\varepsilon^z$ satisfies a Dirichlet problem: $-\Delta
\xi_\varepsilon^z$ is the multiple of the uniform measure on
$\partial B_\varepsilon(z)$ with total mass $2 \pi$ (because $h$ is
$\sqrt{2\pi}$ times the standard GFF; if $h$ were the standard GFF
the total mass would be $1$).  As noted in Section
\ref{ss.GFFnormsection}, this $\xi_\varepsilon^z$ has been computed
explicitly:
\begin{equation}
\label{xi}
\xi_\varepsilon^z(y) = - \log \max\{|z-y|,\varepsilon \} - \tilde G_z(y),
\end{equation}
where $\tilde G_z$ is the harmonic interpolation to $D$ of the first
term on $\partial D$, as long as $B_\varepsilon(z) \subset D$.

Let $\Theta$ be the limit of the measures $\Theta_\varepsilon$ as
$\varepsilon \to 0$: that is, $\Theta$ is the measure on pairs
$(z,h)$ for which the marginal distribution of $z$ is proportional
to $[C(z;D)]^{\gamma^2/2}dz$ and, given $z$, the $\Theta$
conditional law of $h$ is that of the original standard  GFF plus the
deterministic function $\gamma \xi_0^z$ (viewed as a distribution).
For any $\tilde D$ compactly supported on $D$, we will also write $\Theta^{\tilde D}$ for the
measure $\Theta$ conditioned on the event $z \in \tilde D$.  We similarly define $\Theta^{\tilde D}_\varepsilon$
to be $\Theta_\varepsilon$ conditioned on $z \in \tilde D$ (where $\varepsilon$ is less than the distance from $\tilde D$ to $\partial D$).
By the above construction and Proposition \ref{p.browniancircleaverages},
we have that conditioned on $z$, the $\Theta$ law of $h_{e^{-t}}$ is essentially that of a Brownian
motion plus a drift term of $\gamma t$.  This in particular implies the following:

\begin{proposition} \label{p.thickpoint}
With $\Theta$ probability one, $z$ is a {\bf $\gamma$-thick point}
of $h$ by the definition in \cite{huperes}.  That is, $$\liminf_{\varepsilon \to 0} {
h_\varepsilon(z)}/{\log\varepsilon^{-1}} \geq \gamma.$$  In fact,
the limit exists and equality holds almost surely.
\end{proposition}

Since the $\Theta$ marginal law of $h$ is absolutely continuous with
respect to the law of $h$ (with Radon-Nikodym derivative
$\mu_h(D)$), this implies that $\mu_h$ is almost surely supported on
$\gamma$-thick points.  It was shown by Hu, Miller, and Peres that
the set of $\gamma$-thick points has Hausdorff dimension $2 -
\frac{\gamma^2}{2}$ almost surely \cite{huperes}.

\proofof{Proposition \ref{p.hnlimit}} The almost sure weak
convergence of the $\mu^n$ to a limit $\tilde \mu$ is immediate from
the martingale convergence theorem. Recall the expression
(\ref{e.hn}) $$ \mu^n = \exp \left( \gamma h^n(z) -
\frac{\gamma^2}{2}\Var h^n(z) + \frac{\gamma^2}{2}\log C(z; D)
\right)dz,$$ and observe that for each $z$, the exponential term
$$\exp \left( \gamma h^n(z) - \frac{\gamma^2}{2}\Var h^n(z) +
\frac{\gamma^2}{2}\log C(z; D) \right)$$ is a non-negative
martingale with respect to the filtration of $h^n$. (This is a
consequence of (\ref{e.expnormal}).) Fubini's theorem implies that
$\mu^n(A)$ is a martingale for any Borel measurable set $A \subset
D$, and the martingale convergence theorem implies that the limit
$\tilde \mu(A) := \lim \mu^n(A)$ exists almost surely. In particular, this holds
whenever $A$ is a diadic square contained in $D$ and from this
easily follows the desired weak convergence.

We still need to show that $\mu = \tilde \mu$ almost surely, where
$\mu$ is as constructed in Proposition \ref{p.hepsilonlimit}.  It is enough
to show that $\mu(S) = \tilde \mu(S)$ almost surely for each diadic square $S$ compactly supported on $D$,
and since both $\mu$ and $\tilde \mu$ are functions of $h$, this is equivalent to
showing that $\mathbb E[\mu(S)|h] = \mathbb E[\tilde \mu(S)|h]$.  This in turn follows
if we can show $\mathbb E[\mu(S)|h^n] = \mathbb E[\tilde \mu(S)|h^n]$ for all $n$; the
latter expectation is just $\mu^n(S)$ by definition, so it remains only to show \begin{equation}\label{Sexp}\mathbb E[\mu(S)|h^n]  = \mu^n(S).
\end{equation}

To this end, let
$h^n_\varepsilon$ denote the mean value of $h^n$ on $\partial
B_\varepsilon(z)$.  For each particular choice of $z$, and
$\varepsilon$ small enough so that $B_\varepsilon(z) \subset D$, and
for each $n$, we recall from (\ref{bertrandequation}) that $h_\varepsilon(z)$ is a Gaussian random variable
with variance  $\log C(z;D) - \log \varepsilon$ and that $h^n_\varepsilon(z)$ is the conditional expectation of $h_\varepsilon(z)$
given the projection of $h$ to the span of $f_1,f_2, \ldots, f_n$.  Hence
$$\mathbb E[ \varepsilon^{\gamma^2/2}e^{\gamma h_\varepsilon(z)} | h^n] =
\exp \left( \gamma h^n_\varepsilon(z) - \frac{\gamma^2}{2}\Var
h^n_\varepsilon(z) + \frac{\gamma^2}{2}\log C(z; D) \right).$$
Taking the limit as $\varepsilon \to 0$ and using the continuity of
$h^n(z)$ and $\Var h^n(z)$ and the expression (\ref{e.hn}), we have
\begin{equation}\label{e.limmun} \lim_{\varepsilon \to 0} \mathbb E[\mu_\varepsilon(S)| h^n] =
\mu^n(S)\end{equation} for each diadic square $S$.
Using (\ref{e.limmun}) we will have established (\ref{Sexp}) once we show that
\begin{equation} \label{e.limequality}
\mathbb E[\mu(S)|h^n]  :=\mathbb E[\lim_{\varepsilon \to 0} \mu_\varepsilon(S) | h^n]
=\lim_{\varepsilon \to 0} \mathbb E[\mu_\varepsilon(S)| h^n],
\end{equation} provided that $0 \leq \gamma < 2$.

We first argue this in the case $n=0$ and $h^0=0$, that is
\begin{equation} \label{e.limequalityzero}
\mathbb E[\mu(S)]  :=\mathbb E[\lim_{\varepsilon \to 0} \mu_\varepsilon(S)]
=\lim_{\varepsilon \to 0} \mathbb E[\mu_\varepsilon(S)].
\end{equation}
Since Proposition
\ref{p.hepsilonlimit} implies the existence of the limit of $\mu_\varepsilon(S)$, it is
enough to show that the random variables $M_\varepsilon =
\mu_\varepsilon(S)$ are uniformly integrable as $\varepsilon \to 0$.
Let $M = \mathbb E M_\varepsilon$ for $\varepsilon$ small enough so
that $B_\varepsilon(S) \subset D$. (By (\ref{e.expexphepsilon}) this
expectation is the same for all sufficiently small $\varepsilon$.)
The uniform integrability is equivalent to the statement that the
probability measures $\eta_\varepsilon := M^{-1} M_\varepsilon
dM_\varepsilon$ are tight, i.e., for all $\delta$ there exists a
constant $C>0$ such that $\eta_\varepsilon ([C,\infty)) < \delta$
for all $\varepsilon$.  (Here $M^{-1} M_\varepsilon dM_\varepsilon$
denotes the probability measure on $\mathbb R$ whose Radon-Nikodym
derivative with respect to the law of $M_\varepsilon$ is given by
$M^{-1} M_\varepsilon$.)  Since $M_\varepsilon$ is a function of
$h$, this is equivalent to the statement that with respect to the
measure $M^{-1} M_\varepsilon(h) dh$ the random variables
$M_\varepsilon(h)$ are tight.  Recalling
that $M_\varepsilon$ is (up to a constant factor) the Radon-Nikodym derivative of the $h$ marginal of
$\Theta_\varepsilon^S$ with respect to $dh$, this in turn can be rewritten as the
statement that for each $\delta$ we can find a $C$ such that
\begin{equation} \label{e.thetatight} \Theta_\varepsilon^S \{M_\varepsilon(h) > C \} < \delta \end{equation} for all $\varepsilon$.

Throughout the remainder of the proof of (\ref{e.limequalityzero}), all probabilities and expectations will be taken with respect
to $\Theta_\varepsilon^S$.  Let $\varepsilon_0 = \sup \{ \varepsilon' : B_{\varepsilon'}(S)
\subset D \}$.  It suffices to prove the above statement, that for each $\delta$ we can find a $C$ such that (\ref{e.thetatight}) holds, for small $S$ --- precisely, for avoiding boundary effects, we may
restrict attention to $S$ such that $\varepsilon_0$ is larger than the diameter of $S$, in which case $S \subset B_{\varepsilon_0}(z)$
for any $z \in S$.

In order to obtain (\ref{e.thetatight}), we will describe a procedure for sampling from $\Theta_\varepsilon^S$ in multiple stages.
We will then show that (\ref{e.thetatight}) holds when $M_\varepsilon(h)$ is replaced by its conditional expectation given the observations from the first two stages, and the statement we require will follow easily from this.

Precisely, we may sample the pair $(z,h)$ from $\Theta_\varepsilon^S$ in the following steps. In
the first step, we sample $z$ from its marginal law (which is
independent of $\varepsilon$ for $\varepsilon$ sufficiently small).
Write $\tilde h = h - \gamma \xi_\varepsilon^z$.  Given $z$, the $\Theta_\varepsilon^S$ law
of $\tilde h$ is that of a GFF on $D$.  In the second step, we
sample $\mathcal B_t = \tilde h_{e^{-t} \varepsilon_0}(z) - \tilde
h_{\varepsilon_0}(z)$ for all $t \in [0, -\log
(\varepsilon/\varepsilon_0)]$. By Proposition
\ref{p.browniancircleaverages}, $\mathcal B_t$ is a Brownian motion
on this interval independent of $z$.
In the third and final step, we choose $h$ conditioned on the results of the first two steps.

The $\Theta_\varepsilon^S$ conditional expectation of
$\tilde h$ given the whole process $\mathcal B_t$ (which we have
defined only for $t \in [0, -\log (\varepsilon/\varepsilon_0)]$) and
$z$ is given by the function (viewed as a distribution)

$$\tilde h^{\centerdot}(w):= \mathbb E [\tilde h(w) |z, \mathcal B_t
] = \begin{cases}  \mathcal B_{u(w)} & (|z-w| < \varepsilon_0) \\
0 & (|z - w| \geq \varepsilon_0) \end{cases},$$
where
$$
u(w):=
-\log \frac{|z-w| \vee \varepsilon}{\varepsilon_0}.
$$
(We will discuss similar conditional expectations in detail in Section \ref{expmart}.)
{Note that once $z$ is fixed, for each $w$ the mean value of $\tilde
h^{\centerdot}(\cdot)$ on $\partial B_\varepsilon(w)$ (which we denote by
$\tilde h^{\centerdot}_\varepsilon(w)$) is a weighted average of $\mathcal B_t$
over values of $t$ between $u_1(w):=-\log(\varepsilon_1(w)/\varepsilon_0)$ and
$u_2(w):=-\log(\varepsilon_2(w)/\varepsilon_0)$, where
$$
\varepsilon_1(w):=\varepsilon_0 \wedge (|w-z|+\varepsilon),\,\,\, \varepsilon_2(w):=(|w-z|-\varepsilon)\vee \varepsilon ,
$$
with, for $|z-w| \leq \varepsilon_0$, $\varepsilon_1(w)\geq \varepsilon_2(w)$, hence
$u_1(w)\leq u_2(w)$.
From this it is not hard to see
that given $z$ the variance of $\tilde h_\varepsilon^{\centerdot}(w)$ is between
the two values $\{u_1(w),u_2(w)\}$  of $t$.  We claim that each of these bounds differs from the intermediate value
$u(w)$ above, with  $u_1(w)\leq u(w)\leq  u_2(w)$, by at most an additive constant $\log 2$. This is equivalent to
the statement that $\varepsilon_1(w)$ and $\varepsilon_2(w)$
differ from $|z-w| \vee \varepsilon$ by a
multiplicative or inverse factor of at most two, which is easily checked under the further mild assumption that $2\varepsilon \leq
\varepsilon_0$.  Thus the variance of
$\tilde h^{\centerdot}_\varepsilon(w)$ is within $\log 2$ of the value
$$u(w)=-\log \frac{|z-w|}{\varepsilon_0} \wedge -\log \frac{\varepsilon}{\varepsilon_0}.$$}

Since $\mathbb E[\tilde h_\varepsilon(w) | \tilde h^{\centerdot}_\varepsilon(w)] = \tilde h_\varepsilon^{\centerdot}(w)$, we have
\begin{equation}\label{e.varvar} \Var(\tilde h_\varepsilon(w)) =  \mathbb E \Var\bigl(\tilde h_\varepsilon(w)|\tilde h^{\centerdot}_\varepsilon(w)\bigr) + \Var(\tilde h_\varepsilon^{\centerdot}(w)).\end{equation}
Since $\tilde h_\varepsilon(w)$ and $\tilde h^{\centerdot}_\varepsilon(w)$ (both linear functionals of $h$) are jointly Gaussian, the quantity $\Var\bigl(\tilde h_\varepsilon(w)|\tilde h^{\centerdot}_\varepsilon(w)\bigr)$ is in fact independent of $\tilde h^{\centerdot}_\varepsilon(w)$.
Since (as observed above) $|\Var(\tilde h^{\centerdot}_\varepsilon(w)) - u(w)| < \log 2$, we conclude that
$$\bigl|\Var(\tilde h_\varepsilon(w)|\tilde
h^{\centerdot}_\varepsilon(w))- \Var(\tilde h_\varepsilon(w))+ u(w) \bigr| < \log 2,$$
almost surely.

Thus, with respect to $\Theta_\varepsilon^S$, we have
$$\mathbb E [\varepsilon^{\gamma^2/2} e^{\gamma h_\varepsilon(w)}|z, \mathcal B_t]
\asymp  \exp \left(\gamma \tilde h^{\centerdot}_\varepsilon(w) + \gamma^2
\xi^z_\varepsilon(w) - \gamma^2 u(w)/2 \right) \asymp \exp
\left(\gamma \tilde h^{\centerdot}_\varepsilon(w) + \gamma^2 u(w)/2 \right),$$
where we recall that, thanks to (\ref{xi}), $\xi^z_\varepsilon(w)=u(w)-\log \varepsilon_0 -\tilde G_z(w)$,
and where $\asymp$ indicates equality up to a multiplicative factor
bounded between positive constants uniformly in $\varepsilon$ and
$z$.

Now, given any positive constants $a$ and $b$, there is a positive
probability that a Brownian motion $\mathcal B_t$ run for an
infinite amount of time will satisfy $\gamma \mathcal B_t < a + bt$
for all $t \geq 0$. In fact, for each fixed $b$, this probability
can be made as close to one as possible by taking $a$ sufficiently
large. Since $0 \leq \gamma < 2$ we can choose a value of $b$ with
$0 < b < 2-\gamma^2/2$. Then note that conditioned on the event $\mathcal A: \gamma \mathcal B_t < a
+ bt$ for all $t$, and since the mean value
$\tilde h^{\centerdot}_\varepsilon(w)$ is a weighted average of $\mathcal B_t$
over values of $t\in [u_1(w),u_2(w)]$, we have
$\gamma \tilde h^{\centerdot}_\varepsilon(w)<a+b\, u_2(w)\leq a+b\log 2+b\, u(w)$.
 We therefore have, for some constant $C_0$
\begin{eqnarray}\nonumber
\mathbb E [\varepsilon^{\gamma^2/2} e^{\gamma h_\varepsilon(w)}| z, \mathcal B_t,\mathcal A]
&\leq& C_0 \,e^a \exp\left[(b+ \gamma^2/2) u(w) \right],\\
\nonumber
&\leq & C_0\, e^a\, |z-w|^{- b - \gamma^2/2}
\end{eqnarray}
for $|z-w| <
\varepsilon_0$. Since $S \subset B_{\varepsilon_0}(z)$ for $z \in S$, this in turn implies that
\begin{eqnarray}\nonumber\mathbb
E[\mu_\varepsilon(S)| z, \mathcal B_t,\mathcal A] &\leq&
\int_{B_{\varepsilon_0}(z)} C_0\, e^a\, |z-w|^{-b - \gamma^2/2} dw,
\end{eqnarray}
and since $b+\gamma^2/2 < 2$,
the right hand side
is at most a finite constant $C_1 = C_1(a)$ that is independent of
$\varepsilon$. Now, given $b$ and a constant $\delta>0$ we can
choose $a$ large enough so that the probability {of the event $\mathcal A$ (that $\gamma
\mathcal B_t < a + bt$ for all $t$)} is at least $1 - \delta/2$. Then
we take $C = \frac{C_1(a)}{\delta/2}$. If there were probability at
least $\delta$ that $\mu_\varepsilon(S)
> C$ then there would have to be probability at least $\delta/2$
that event $\mathcal A$
{\em and}
$\mu_\varepsilon(S)
> C$ simultaneously happen, which would contradict our bound on the conditional
expectation of $\mu_\varepsilon(S)$ given $\mathcal A$.
 This implies that the probability measures
$\eta_\varepsilon$ are tight, which in turn completes the proof of
(\ref{e.limequalityzero}), which is the $n=0$ and $h^0=0$ case of (\ref{e.limequality}).

{As a tool, we used heavily within
the proof of (\ref{e.limequalityzero}) the probability measure $\Theta^S_\varepsilon$ on $(z,h)$
pairs.  Extending (\ref{e.limequality}) to the case $n \not = 0$ does not require this tool;
in the discussion below we will use only the original $dzdh$ measure.}
Note that since the random variables $\mu_\varepsilon(S)$ converge $dzdh$ almost surely to a limit
(with expectation $\lim_{\varepsilon \to 0} \mathbb E
\mu_\varepsilon(S)$), it must be the case that conditioned on $h^n$
(for almost all values of $h^n$), we still have that $\mu_\varepsilon(S)$
converges $dzdh$ almost surely to a limit.  The fact that \begin{equation}\label{fatou} \mathbb
E[\lim_{\varepsilon \to 0} \mu_\varepsilon(S)| h^n] \leq
\lim_{\varepsilon \to 0} \mathbb E[\mu_\varepsilon(S)| h^n]\end{equation} for
almost all $h^n$ is immediate from Fatou's lemma.  From the unconditional
result, we know that equality holds when we integrate over possible values of $h^n$ ---
hence equality must hold in (\ref{fatou}) for almost all $h^n$.
The extension of (\ref{e.limequality}) to non-zero $h^0$ is trivial
for functions that are piecewise constant on diadic squares, and the
more general case follows easily by approximation by piecewise
constant functions.

Proposition \ref{p.hnlimit} is an immediate consequence of
(\ref{e.limmun}) and (\ref{e.limequality}).\qed

\section{KPZ proofs}
\subsection{Circle average KPZ}\label{expmart}
{For fixed $z \in D$, choose some radius $\varepsilon_0$ such that
$B_{\varepsilon_0}(z) \subset D$. As a first step, we estimate the expectation of
the quantum measure $\mu_h({B_\varepsilon(z)})$,
\textit{given} the difference of
circle averages $h_\varepsilon(z)-h_{\varepsilon_0}(z)$ for $\varepsilon \leq \varepsilon_0$.
Recalling the notation of Proposition
\ref{p.hnlimit}, we take $h^0=0$, $n=1$, and
\begin{equation}\label{f1}
f_1 = \big(\xi^z_\varepsilon-\xi^z_{\varepsilon_0}\big)/
||\xi^z_\varepsilon-\xi^z_{\varepsilon_0}||_\nabla .
\end{equation}
Recall from (\ref{xixi}) that the square Dirichlet norm of function $\xi^z_\varepsilon$ (\ref{defxi}) is
such that
$||\xi^z_\varepsilon||_\nabla^2=(\xi^z_\varepsilon, \xi^z_\varepsilon)_\nabla=
\xi^z_\varepsilon(z)$, and from Proposition \ref{p.hepsiloncovariance} that
$(\xi^z_\varepsilon,\xi^z_{\varepsilon_0})_\nabla=\xi^z_{\varepsilon_0}(z)$. One thus finds
$||\xi^z_\varepsilon-\xi^z_{\varepsilon_0}||^2_\nabla=-\log(\varepsilon/\varepsilon_0)$ and
\begin{eqnarray}\label{xi-xi0}
\big(\xi^z_\varepsilon-\xi^z_{\varepsilon_0}\big)(y)=\begin{cases}
-\log(\varepsilon/\varepsilon_0),\,\,\,\,\,\,\,\,\,\,\,\,\,\,\,\,\,0\leq |y-z| \leq \varepsilon\\
-\log(|y-z|/\varepsilon_0),\,\,\, \varepsilon \leq |y-z| \leq \varepsilon_0\\
\,\,\,\,\,\,0,\,\,\,\,\,\,\,\,\,\,\,\,\,\,\,\,\,\,\,\,\,\,\,\,\,\,\,\,\,\,\,\,\,\,
\varepsilon_0\leq |y-z|.
\end{cases}
\end{eqnarray}
The projection $h^1$ of $h$ onto the span of $f_1$ and its variance are then
\begin{eqnarray}\label{projectionbis}
h^1(y)&=&\big[h_\varepsilon(z)-h_{\varepsilon_0}(z)\big]\frac{\big(\xi^z_\varepsilon-\xi^z_{\varepsilon_0}\big)(y)}
{-\log(\varepsilon/\varepsilon_0)},\\ \label{varprojectionbis}
\Var\, h^1(y)&=& \frac{\big(\xi^z_\varepsilon-\xi^z_{\varepsilon_0}\big)^2(y)}{-\log(\varepsilon/\varepsilon_0)}
,
\end{eqnarray}
where we recall that $\Var\, [h_\varepsilon(z)-h_{\varepsilon_0}(z)]=-\log(\varepsilon/\varepsilon_0)$.}

{Recalling the notation of Proposition \ref{p.hepsilonlimit}, the conditional
expectation formula (\ref{condexpect}) for $\mu$  in Proposition \ref{p.hnlimit} gives
\begin{equation}\label{muBepsbis}\mathbb E_h \left[\int_{B_\varepsilon(z)} e^{\gamma h} dz |
h_\varepsilon(z)-h_{\varepsilon_0}(z) \right]=\mathbb E_h \left[\mu_h({B_\varepsilon(z)}) |h_\varepsilon(z)-h_{\varepsilon_0}(z)
\right]=\mu^1\left(B_\varepsilon(z)\right),
\end{equation}
where $\mu^1$ is the projected measure (\ref{e.hn})
\begin{equation}\label{e.hn1}
\mu^1(dy)=\exp\left(\gamma h^1(y)-\frac{\gamma^2}{2}\Var\, h^1(y)+\frac{\gamma^2}{2}\log C(y;D)\right)dy.
\end{equation}
Note that by (\ref{xi-xi0}), $h^1(y)$ does not depend on $y$ for $y\in B_\varepsilon(z)$
\begin{eqnarray}
\nonumber
h^1(y)&=&h_\varepsilon(z)-h_{\varepsilon_0}(z),\,\,\,y \in B_\varepsilon(z),
\\
\nonumber
\Var\, h^1(y)&=& -\log(\varepsilon/\varepsilon_0).
\end{eqnarray}
We therefore have
\begin{eqnarray}\label{mu1}
\mu^1(dy)&=&\mu^0(dy)
\left(\frac{\varepsilon}{\varepsilon_0}\right)^{\gamma^2/2}
\exp\big[h_\varepsilon(z)-h_{\varepsilon_0}(z)\big],\,\,\,y\in B_\varepsilon(z), \\ \label{mu0}
\mu^0(dy)&:=&\big[C(y;D)\big]^{\gamma^2/2}dy.
\end{eqnarray}
Define the ($\gamma$-dependent) average $C_\varepsilon(z;D)$ of the conformal radius over the ball
$B_\varepsilon(z)$ via
the average moment
\begin{equation}\label{avC}
\big[C_\varepsilon(z;D)\big]^{\gamma^2/2}:=\frac{\mu^0\big(B_\varepsilon(z)\big)}{\mu_0\big(B_\varepsilon(z)\big)}=
\frac{1}{\pi \varepsilon^2}\int_{B_\varepsilon(z)}\big[C(y;D)\big]^{\gamma^2/2}dy,
\end{equation}
so that for $\varepsilon \to 0$
$$
\lim_{\varepsilon\to 0} C_{\varepsilon}(z;D)=C(z;D).
$$
We then have the simple expression
\begin{equation}\label{mu1Bepsilon}
\mu^1\big(B_\varepsilon(z)\big)=\pi \varepsilon^{\gamma Q}\left(\frac{C_\varepsilon(z;D)}{\varepsilon_0}
\right)^{\gamma^2/2}\exp\big[h_\varepsilon(z)-h_{\varepsilon_0}(z)\big],
\end{equation}
where, as above, $Q={2}/{\gamma}+{\gamma}/{2}$.}

{As an alternative, one may wish to estimate the expectation of the quantum measure $\mu_h({B_\varepsilon(z)})$,
\textit{given only} the
circle average $h_\varepsilon(z)$. In the notation of Proposition \ref{p.hnlimit}, we
take in that case $h^0=0$, $n=1$, and $\tilde f_1 = \xi^z_\varepsilon/|| \xi^z_\varepsilon||_\nabla$,
with the square Dirichlet norm $|| \xi^z_\varepsilon||_\nabla^2=( \xi^z_\varepsilon, \xi^z_\varepsilon)_\nabla=
\xi^z_\varepsilon(z)$.
The projection ${\tilde h}^1$ of $h$ onto the span of $f_1$ and its variance are then
\begin{eqnarray}\label{projection}
{\tilde h}^1(y)&=&h_\varepsilon(z)\frac{\xi^z_\varepsilon(y)}{\xi^z_\varepsilon(z)},\\ \label{varprojection}
\Var\, {\tilde h}^1(y)&=&\Var\, h_\varepsilon(z) \left(\frac{\xi^z_\varepsilon(y)}{\xi^z_\varepsilon(z)}\right)^2=\frac{\big(\xi^z_\varepsilon(y)\big)^2}{\xi^z_\varepsilon(z)},
\end{eqnarray}
where we recall that $\Var\, h_\varepsilon(z)=\xi^z_\varepsilon(z)=-\log\left(\varepsilon/C(z;D)\right)$.}

{The conditional
expectation formula (\ref{condexpect}) for $\mu$  in Proposition \ref{p.hnlimit} gives in this case
\begin{equation}\label{muBeps}\mathbb E_h \left[\int_{B_\varepsilon(z)} e^{\gamma h} dz |
h_\varepsilon(z) \right]=\mathbb E_h \left[\mu_h({B_\varepsilon(z)}) |h_\varepsilon(z)
\right]={\tilde \mu}^1\left(B_\varepsilon(z)\right),
\end{equation}
where ${\tilde \mu}^1$ is the projected measure (\ref{e.hn})
\begin{equation}\label{e.tildehn1}
{\tilde \mu}^1(dy)=\exp\left(\gamma \tilde h^1(y)-\frac{\gamma^2}{2}\Var\, \tilde h^1(y)
+\frac{\gamma^2}{2}\log C(y;D)\right)dy.
\end{equation}
Note that when $y \in B_\varepsilon(z)$, $\xi^z_\varepsilon(y)=-\log\varepsilon-\tilde G_z(y)$, so that the difference
$\xi^z_\varepsilon(z)-\xi^z_\varepsilon(y)=\log C(z;D)+\tilde G_z(y)$ is harmonic in $y$ and its modulus is
equivalent to $\varepsilon |\tilde G_z'(z)|$ for $\varepsilon$ small, where $\tilde G_z'(z)$
is the derivative at $z$ of the harmonic extension $\tilde G_z$. It follows that in ball $B_\varepsilon(z)$,
$\xi^z_\varepsilon(y)/\xi^z_\varepsilon(z)=1+O(\varepsilon/\log\varepsilon)$.  Lastly, the function $C(y;D)$ is real
analytic. Hence from (\ref{projection}),
(\ref{varprojection}) and (\ref{e.tildehn1}) above, it follows from (\ref{muBeps}) that
 for $\varepsilon \to 0$
\begin{eqnarray}  \label{ExpectedQuantumArea}
\mathbb E \left[\mu({B_\varepsilon(z)}) |h_\varepsilon(z)\right]={\tilde \mu}^1\left(B_\varepsilon(z)\right)\simeq \mu_{\odot}\left(B_\varepsilon(z)\right),
\end{eqnarray}
where $\mu_\odot$ is defined as
\begin{eqnarray} \label{muodot}
 \mu_{\odot}\big(B_\varepsilon(z)\big)
:=\pi \varepsilon^{\gamma Q}e^{\gamma h_\varepsilon(z)},\,\,\,Q=2/\gamma+\gamma/2,
\end{eqnarray}
in the sense that the ratio of the  two quantities
tends to $1$ as $\varepsilon \to 0$.
Note that $\mu_{\odot}$ is not a measure, but simply a quantity defined on balls of the form $B_\varepsilon(z)$.}
  {Notice then that the first conditional measure $\mu^1\big(B_\varepsilon(z)\big)$ (\ref{mu1Bepsilon})
can also be written as
\begin{equation}\label{mu1muodot}
\mu^1\big(B_\varepsilon(z)\big)=\pi \varepsilon_0^2\, C_\varepsilon(z;D)^{\gamma^2/2}
\frac{\mu_{\odot}\big(B_\varepsilon(z)\big)}{\mu_{\odot}\big(B_{\varepsilon_0}(z)\big)}.
\end{equation}}

{For any $\varepsilon \leq
\varepsilon_0$ define then
\begin{eqnarray}
\label{tdef}
t &:=& -\log(\varepsilon/\varepsilon_0)\\
\label{Vtdef}
V_t&:=& h_\varepsilon(z) - h_{\varepsilon_0}(z).
\end{eqnarray}
 The law of $V_t$ is that
of a Brownian motion with $V_0 = 0$ (by Proposition
\ref{p.browniancircleaverages}).  We can then rewrite (\ref{mu1Bepsilon}) as
\begin{equation}\label{mu1BepsilonVt}
\mu^1\big(B_\varepsilon(z)\big)=\pi \varepsilon_0^{2}\,C_\varepsilon(z;D)^{\gamma^2/2}e^{\gamma V_t - \gamma Q t}.
\end{equation}
Similarly, we can rewrite (\ref{muodot}) identically as
\begin{equation}\label{mudotmu0dot}
 \mu_{\odot}\big(B_\varepsilon(z)\big)=\mu_{\odot}\big(B_{\varepsilon_0}(z)\big)\,e^{\gamma V_t - \gamma Q t},
\end{equation}
 in accordance with (\ref{mu1muodot}). In the expression (\ref{mu1BepsilonVt}) for the measure $\mu^1$, the first non constant factor is the same as
\eref{avC}, which is a slowly
 varying, deterministic function of $z$ (and of $\varepsilon$),
whereas in the expression (\ref{mudotmu0dot}) for $\mu_\odot$,
the first factor is the quantity $\mu_{\odot}\big(B_{\varepsilon_0}(z)\big)$, which is the exponential of a centered Gaussian
variable, $h_{\varepsilon_0}(z)$, whose variance, $-\log\left(\varepsilon_0/C(z;D)\right)$, varies slowly with $z$.
In both expressions, the latter factor is the
exponential of a simple Brownian motion with drift, and
is  \textit{independent} of $z$.}
\begin{definition} \label{Btilde} Let $\tilde B^\delta(z)$ be the largest Euclidean ball in $D$
centered at $z$ for which $e^{\gamma V_t - \gamma Q t}$ is equal to
$\delta$.  The radius of this ball is $e^{-T_A}$ where $$T_A := \inf
\{t: -V_t + Q t = A \},$$ and $A:= -(\log \delta) / \gamma$.
\end{definition}
As a step towards Theorem \ref{t.verystrongquantumKPZ} we prove the
following in this section, which is perhaps the most straightforward
form of KPZ to prove:

\begin{theorem} \label{t.expectationquantumKPZ}
Theorem \ref{t.verystrongquantumKPZ} holds with $B^\delta(z)$
replaced with $\tilde B^\delta(z)$.  That is, in the setting of
Theorem \ref{t.verystrongquantumKPZ}, if
$$\lim_{\varepsilon \to 0} \frac{ \log \mathbb E \mu_0 \{z: B_\varepsilon(z) \in \mathcal X\} }{\log \varepsilon^2} = x,$$
then it follows that, when $\mathcal X$ and $\mu$ are chosen
independently, we have
$$\lim_{\delta \to 0} \frac{ \log \mathbb E \mu \{z: \tilde B^\delta(z) \in \mathcal X\} }{\log \delta} = \Delta,$$
where $\Delta$ is the non-negative solution to $$ x =
\frac{\gamma^2}{4} \Delta^2 + \left( 1 -
\frac{\gamma^2}{4}\right)\Delta.
$$
\end{theorem}

We present two proofs: the first based on exponential martingales,
the second based on large deviations theory and Schilder's theorem.
(The first proof is shorter, but readers familiar with large
deviations of Brownian motion will recognize that it is essentially
the second proof in the disguise.)

Both proofs use the fact that $$\mathbb E_h \; \mu \{z: \tilde
B^\delta(z) \in \mathcal X\}$$ is proportional to $$\Theta \{(z,h):
\tilde B^\delta(z) \in \mathcal X \},$$ to replace an expectation
computation with a probability computation. (Recall the definition
of $\Theta$ from Section \ref{s.rootedmetrics}.) While this
rephrasing is not strictly necessary for the expectation computation
below, it is conceptually quite natural.

We use the definitions (\ref{tdef}) and (\ref{Vtdef}) of $V_t$ given above, and assume that the
fixed $\varepsilon_0$ is smaller than the distance from $\tilde D$
(recall that this was the compact subset of $D$ in Theorem \ref
{t.verystrongquantumKPZ}) to $\partial D$.

As mentioned in Section \ref{s.rootedmetrics}, the $\Theta$
conditional law of $h$ given $z \in D$ is that of the original GFF
plus the deterministic function $- \gamma \log |z-y|-\gamma\tilde G_z(y)$.
 {The $\Theta$
conditional law of the circular average $h_\varepsilon(z)$ is then that of the original GFF
circular average plus $-\gamma\log\varepsilon +\gamma\log C(z;D)$.} Thus (for $z$
restricted to points of distance at least $\varepsilon_0$ from
$\partial D$) the $\Theta$ conditional law of (\ref{Vtdef})
{$V_t=h_{\varepsilon_0e^{-t}}(z)-h_{\varepsilon_0}(z)$} given $z$ is
that of ${\mathcal B}_t + \gamma t$, with $t=-\log(\varepsilon/\varepsilon_0)$, and
where ${\mathcal B}_t$ evolves
as a standard Brownian motion---in particular, $z$ is independent of
the process $V_t$.


\begin{proof}
The $\Theta$ law of $T_A$ is that of \begin{equation}
\label{stoppingtime} \inf\{t: {\mathcal B}_t + at=A
=-(\log \delta)/\gamma\},\, a:=Q-\gamma
=\frac{2}{\gamma}-\frac{\gamma}{2}> 0,
\end{equation}
where $(\pm){\mathcal B}_t$ is standard Brownian motion with
${\mathcal B}_0=0.$  Let $q_A$ be the $\Theta$ probability that the ball of radius $e^{-T_A}$ centered at $z$
is in $\mathcal X$.  Since $z$ is independent of $T_A$, the theorem
hypothesis implies that conditioned on $T_A$, the probability that the ball of radius $e^{-T_A}$ centered at $z$ is in $\mathcal
X$ is approximately $\exp{(-2xT_A)}$, in the sense that the ratio of
the logs of these two quantities tends to $1$ as $T_A \to \infty$.
Computing the expectation
\begin{equation}\label{e.exact}
\mathbb E \left[\exp{(-2xT_A)}\right],\end{equation} with respect to
a random $T_A$ will give us upper and lower bounds on $q_A$ since it
easily follows that
\begin{equation} \label{e.xcontinuity} \mathbb E
\left[\exp{(-2x_1T_A)}\right] \leq q_A \leq \mathbb E
\left[\exp{(-2x_2T_A)}\right],\end{equation} for any fixed $0< x_2 <
x < x_1$ and sufficiently large $A$.

To compute (\ref{e.exact}), consider for any $\beta$ the exponential
martingale $\exp( \beta {\mathcal B}_t - \beta^2 t/2 )$.
{Since $a>0$,
the stopping time $T_A$ is finite a.s. Since ${\mathcal B}_t+at\leq A$ for $t\in[0,T_A]$, the argument of the
exponential,
$\beta {\mathcal B}_t - \beta^2 t/2$, stays bounded from above, for $\beta \geq 0$, by
$\beta A-(\beta a+\beta^2/2)t\leq \beta A$, hence by a fixed constant.}
One can thus apply the exponential martingale at the stopping time $T_A <\infty$
$${\mathbb E} \left[ \exp(\beta {\mathcal B}_{T_A} - \beta^2 T_A/2) \right]=1.$$
By definition ${\mathcal B}_{T_A}=A-aT_A$. Thus,
$${\mathbb E} \exp[- (\beta a+\beta^2/2) T_A]=\exp(-\beta A).$$
Setting $2x:=\beta a+\beta^2/2$, we obtain
\begin{equation} \label{e.expmartingalekpz} {\mathbb E}
\exp(- 2x T_A) =\exp(-\beta A)=\delta^{\beta/\gamma}
.\end{equation} Now if we set $\Delta = \beta/\gamma$, and $a =
Q-\gamma = \frac{2}{\gamma} - \frac{\gamma}{2}$, we find that the
equation $2x:=\beta a+\beta^2/2$, with $\beta \geq 0$, is equivalent to the KPZ formula.
The continuity of this expression and (\ref{e.xcontinuity}) together
yield the theorem. \qed
\end{proof}

We remark that the above yields the explicit probability
distribution $P_A(t)$. The inverse Laplace transform $P_A(t)$ of
$f_A(x):={\mathbb E} \exp(- 2x T_A)$, with respect to $2x$, is the
probability density such that $P_A(t) dt:= {\textrm {Prob}}
\left(T_A\in [t, t+dt]\right)$. Its explicit expression is \cite{HandbookBrownian}
\begin{equation} \label{PA} P_A(t)=(2\pi)^{-1/2} At^{-3/2} \exp\left[-({1}/{2}) \left(A
t^{-1/2}-a t^{1/2}\right)^2\right],\end{equation} where as above we
have $A=-(\log \delta)/\gamma$, $t=-\log (\varepsilon/\varepsilon_0)$ and
$a=Q-\gamma$.

\subsection{Large deviations proof of circle average KPZ}
In this section, we present an alternative proof of Theorem
\ref{t.expectationquantumKPZ}, using Schilder's theorem.

\begin{lemma} \label{l.hittingldp} Fix a constant $a > 0$.  Let ${\mathcal B}_t$ be a standard Brownian
motion. For each $A > 0$, write
\begin{equation}
\label{stoppingtime2} T_A = \inf \{t: {\mathcal B}_t + a t = A\}.
\end{equation}
 Then the family of random variables $A^{-1} T_A$ satisfies a large
deviations principle with speed $A$ and rate function
$$I(\eta) = \frac{\eta}{2} \left(\frac{1}{\eta} - a \right)^2  = \frac{\eta^{-1}}{2} - a + a^2\frac{\eta}{2}.$$
\end{lemma}

\proof Schilder's Theorem (see Theorem 5.3.2 of \cite{DZ}) gives an
LDP for the sample path of $\alpha^{-1} {\mathcal B}_t$ (where
${\mathcal B}_t$ is standard Brownian motion) with speed $\alpha^2$
and rate function given by the Dirichlet energy. The variable
$A^{-1} T_A$ can be written as $\inf \{t:W_t + at = 1 \}$ where $W_t
= {\mathcal B}_{At}/A$, which has the same law as $\sqrt{A^{-1}}
{\mathcal B}_t$. Clearly, among all functions $\phi \in
H_1([0,\infty))$ satisfying $\phi(0) = 0$ and $\inf \{t :\phi(t) +
at = 1 \} \leq \eta$, the one with minimal Dirichlet energy is given
by $$\phi(t) =
\begin{cases} (\frac{1}{\eta} - a) t & t < \eta \\
(\frac{1}{\eta} - a) \eta & t \geq \eta. \\
\end{cases}
$$
By the contraction principle (Theorem 4.2.1 of \cite{DZ}), the rate
function desired in Lemma \ref{l.hittingldp} is given by this
minimal Dirichlet energy, i.e., $I(\eta) = \eta (\frac{1}{\eta} - a
)^2/2.$ \qed

\begin{lemma} \label{l.hittingldp2} Consider the following two part experiment.  First choose $T_A$ as
above.  Then toss a coin that comes up heads with probability $$e^{-
2x T_A}.$$  Then the probability that the coin comes up heads decays
exponentially in $A$ at rate $\beta$ where $\beta$ and $x$ are
related by \begin{equation} \label{KPZ} \beta = \inf_\eta
\left\{I(\eta) + 2x \eta\right\}, \end{equation} or equivalently by
\begin{equation} \label{KPZ2} 4x = \beta^2 + 2 a \beta.
\end{equation}
\end{lemma}

\proof The exponential decay with the exponent given in (\ref{KPZ})
is an immediate consequence of Varadhan's integral lemma (Theorem
4.3.1 of \cite{DZ}). To derive (\ref{KPZ2}) from (\ref{KPZ}), we set
the derivative of $I(\eta) + 2x \eta$ to zero and find $-\eta^{-2}/2
+ a^2/2 + 2x = 0$.  Hence the minimum is achieved at
\begin{equation} \label{e.ldpmin} \eta_0 = (a^2 + 4x)^{-1/2}.\end{equation} We then compute $\beta =
I(\eta_0) + 2x \eta_0$ to be
$$ (a^2 + 4x)^{1/2}/2 - a + a^2 (a^2 + 4x)^{-1/2}/2 + 2x (a^2 + 4x)^{-1/2}.$$
Simplifying, we have $\beta =(a^2 + 4x)^{1/2} - a,$ which is
equivalent to (\ref{KPZ2}). \qed

\proofof{Theorem \ref{t.expectationquantumKPZ}}
As above, we aim to show that $P\{\tilde B^\delta(z) \in \mathcal
X\}$ scales as $e^{-\beta A} = \delta^{\beta/\gamma} =
\delta^{\Delta}$ where $\Delta = \beta/ \gamma$, where $\delta$ and
$\varepsilon$  are related via the stopping time $T_A$
(\ref{stoppingtime}). Rescaling $T_A$ by $A^{-1}$ as in
(\ref{stoppingtime2}) puts us in the framework of large deviations
Lemma \ref{l.hittingldp}. As above, to describe the probability
$P\{\tilde B^\delta(z)\in \mathcal X\}$ we can imagine that we first
choose the radius $\varepsilon$ of $\tilde B^\delta(z)$ and then
toss a coin that comes up heads with probability $\varepsilon^{2x}$
to decide whether the ball is in $\mathcal X$.  This puts us in the
framework of the second large deviations Lemma \ref{l.hittingldp2}.
Using (\ref{KPZ2}), we have
$$4x = \beta^2 + 2a \beta = (\gamma \Delta)^2 + 2a\gamma \Delta,$$
where $a = Q-\gamma$. Plugging in this value of $a$ and simplifying,
we obtain the KPZ relation

$$x = \frac{1}{4} \left(\gamma^2 \Delta^2 + 2 \gamma(Q - \gamma)
\Delta \right) = \frac{\gamma^2}{4} \Delta^2 + \left( 1 -
\frac{\gamma^2}{4}\right)\Delta.$$

As in the previous proof, if the probability given $\varepsilon$ is
not exactly $\varepsilon^{2x}$, but the ratio of the log of this
probability to the log of $\varepsilon^{2x}$ tends to $1$ as
$\varepsilon \to 0$, we obtain the same theorem by using alternate
values of $x$ to give upper and lower bounds. \qed

The optimum $\eta_0 = (a^2 + 4x)^{-1/2}$ obtained in
(\ref{e.ldpmin}) has a natural interpretation --- it suggests that
(in the large deviations sense described above) $T_A/A$ is
concentrated near $\eta_0$.

Equivalently, since
$$\Delta = \frac{\beta}{\gamma} =
\frac{(a^2+4x)^{1/2} - a} \gamma,$$ we can say that $A/T_A$ is
concentrated near $\gamma \Delta + a = \gamma \Delta + Q - \gamma$,
which implies that $\frac{\log \delta}{\log \varepsilon}$ is
concentrated near $\gamma (\gamma \Delta + Q - \gamma)$. Note that
the same result can also be obtained directly from the explicit
probability density (\ref{PA}). This is the concentration one
obtains at an $\alpha$-thick point of the GFF $h$, where
\begin{equation} \label{e.alphathick} \alpha = \gamma - \gamma
\Delta. \end{equation} Very informally, this suggests that the quantum
support of a quantum fractal of dimension $\Delta$ is made up of
$\alpha$-thick points of $h$. This generalizes the idea of
Proposition \ref{p.thickpoint}, which concerns the case $\Delta=0$.

{\subsection{Tail estimates for quantum measure}}
\begin{lemma} \label{l.totalmass}
Let $D=\D = B_1(0)$ be the unit disc and fix $\gamma \in [0,2)$ and
take $\mu = e^{\gamma h(z)}dz$ as defined previously.  Then the
random variable $A = \log \mu(B_{1/2}(0))$ satisfies $p_A(\eta):=
\mathbb P[A < \eta] < e^{-C \eta^2}$ for some fixed constant $C>0$
and all sufficiently negative values of $\eta$.
\end{lemma}

\proof Let $h'$ be the projection of $h$ onto the space of functions
in $H(\D)$ that are harmonic inside the two discs $B_{1/4}(1/4)$ and
$B_{1/4}(-1/4)$.  {(See Figure \ref{balls}.)}  Recall that the orthogonal complement of this
space is the space of functions supported on these discs, or more
precisely, the space $H[B_{1/4}(1/4) \cup B_{1/4}(-1/4)]$.  Hence,
the law of $h-h'$ is that of a sum of independent Gaussian free
fields on $B_{1/4}(1/4)$ and $B_{1/4}(-1/4)$ with zero boundary
conditions {(see, e.g., \cite{MR2322706})}.

Let $\underline h$ be the infimum of $h'$ over the union of the two
smaller discs $B_- := B_{1/8}(-1/4)$ and $B_+ := B_{1/8}(1/4)$. Write
$A_- = \log \mu_{h-h'}(B_-)$ and $A_+ = \log \mu_{h-h'}(B_+)$. By
Proposition \ref{Qtransformation} the law of each of $A_+$ and $A_-$
is the same as the law of $A + \gamma Q \log(1/4) = A - \gamma Q\log
4$; clearly $A_+$ and $A_-$ are independent of one another. Also,
$\mu_h (B_+) \geq e^{\gamma \underline h} \mu_{h-h'}(B_+)$ (and
similarly for $B_-$), which implies \begin{equation}
\label{A1A2bound} A \geq \max \{A_-, A_+ \} + \gamma \underline
h.\end{equation}

\begin{figure}
\begin{center}
\includegraphics[scale = .3]{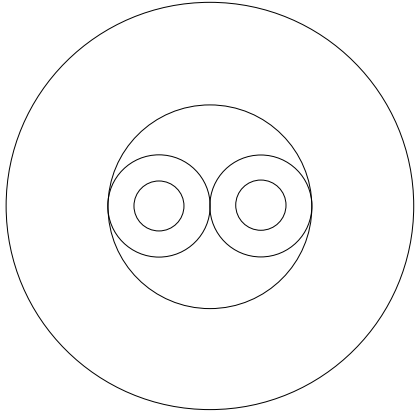}
\end{center}
\caption{\label{balls} The balls $B_1(0)$, $B_{1/2}(0)$, $B_{1/4}(\pm 1/4)$, and $B_{\pm} := B_{1/8}(\pm 1/4)$.}
\end{figure}

First we will show that the probability distribution of $\underline
h$ has superexponential decay. Since $h'$ is harmonic on $B_+$ (with
$h'(1/4) = h'_{1/8}(1/4)$) this $h'$ is the real part of an analytic
function on $B_+$.  In particular, $h'$ restricted to $B_+$ can be
expanded as $h'(1/4) + \sum_{n=1}^\infty \Re [a_n 4^n(z-1/4)^n]$ for
some complex $a_n$. Since each of the random variables $\Re a_n$ and
$\Im a_n$ is a real-valued linear functional of $h'$, it is a
Gaussian random variable. The variance of the latter can be estimated as follows.



{Under the conformal map $\varphi$ such that $\varphi(1/4)=0$ and $\varphi\left(B_{1/4}(1/4)\right)=B_1(0)$,
 the original domain $\mathbb D$ is mapped onto a new domain $D=\varphi(\mathbb D)$.}
  {Let us define on $\mathbb C$ the set of real functions
\begin{eqnarray}\phi_n(z)&:=&\Re [\Phi_n(z)],\,\,\, \psi_n(z):=\Im [\Phi_n(z)]\\
\Phi_n(z)&:=&\begin{cases}  \bar z^n /{(\pi n)^{1/2}}\,\,\,\, (|z| \leq 1)\\    z^{-n}/{(\pi n)^{1/2}}\, (|z| \geq 1) \end{cases}
\end{eqnarray}
The functions $\Re z^n$ and
$\Im z^n$ have on $\D$  the Dirichlet energy
\begin{equation}
\label{energy}
\int_\D n^2 |z^{n-1}|^2 dz = n^2 \int_0^1 r^{2n-2}2\pi rdr =  \pi
n,
\end{equation}
so that the set $\{\phi_n, \psi_n\}$ obeys the orthogonality relations in $\mathbb D$
$$(\phi_m,\phi_n)^{\mathbb D}_\nabla=(\psi_m,\psi_n)^{\mathbb D}_\nabla=\, \delta_{m,n};\,\,(\phi_m,\psi_n)^{\mathbb D}_\nabla=0.$$
From the conformal invariance of the Dirichlet inner product,  and from the expansion
$$h'(z)=h'(0) + \sum_{n=1}^\infty \Re a_n\, \phi_n(z) +\Im a_n\,  \psi_n(z)$$
we obtain the explicit form of the coefficients $a_n$
$$\Re a_n=(h',\phi_n)^{\mathbb D}_\nabla,\,\,\,\Im a_n=(h',\psi_n)^{\mathbb D}_\nabla ,$$
where after the conformal map $\varphi$, $h'$ is now understood as the projection on $\textrm{Harm}\, \mathbb D$ of the GFF $h$ with zero boundary conditions on  $\partial D$. This can be rewritten in the complex form
$$
a_n=(h', \Phi_n)^{\mathbb D}_\nabla=( h, \Phi_n)^{\mathbb D}_\nabla ,
$$
where use was made of the orthogonality of $h-h'$ and  $\Phi_n$.
By inversion with respect to the unit circle $\partial \mathbb D$, this is also $( h, \Phi_n)^{\mathbb C\setminus \mathbb D}_\nabla$, so that
$$ a_n=\frac{1}{2}( h, \Phi_n)^{\mathbb C}_\nabla,$$
where now the Dirichlet inner product extends to the whole plane.
Since $h$ vanishes outside of  $D$, we can also write
$$ (h, \Phi_n)^{\mathbb C}_\nabla=(h,\tilde {\Phi}_n)^{D}_\nabla$$
where $\tilde{\Phi}_n:={\Phi}_n-\Phi^H_n$, $\Phi^H_n$ being the harmonic extension to $D$ of ${\Phi}_n$ restricted to $\partial D$.
Specifying this separately for the real and imaginary parts of $a_n$ and $\Phi_n$, we have
$$\Re a_n=\frac{1}{2}(h,\tilde {\phi}_n)^{D}_\nabla ,\,\,\,\,  \Im a_n=\frac{1}{2}(h,\tilde { \psi}_n)^{D}_\nabla ,$$
where $\tilde{\phi}_n:={\phi}_n-\phi^H_n$, with a similar definition  for the imaginary component.
Since $\tilde{\phi}_n$ is  in  $H(D)$, we have
$$\Var\,  \Re a_n=\frac{1}{4}\Var (h,\tilde {\phi}_n)^{D}_\nabla=\frac{1}{4}(\tilde {\phi}_n,\tilde {\phi}_n)^{D}_\nabla ,$$
together with an entirely similar expression for $\Var\,  \Im  a_n$.
 We now wish to argue that $$(||\tilde {\phi}_n||^{D}_\nabla)^2 \leq (||\tilde {\phi}_n||^{\mathbb C}_\nabla)^2 \leq (|| {\phi}_n||^{\mathbb C}_\nabla)^2 =\frac{2}{\pi n}.$$
The first inequality is obvious, while the second one is  a consequence of the orthogonal decomposition $\phi_n=\tilde \phi_n + \phi^{H}_n$ on $\mathbb C$.
We thus conclude that the variances $\Var\, \Re a_n$ and $\Var\, \Im a_n$ are at most ${1}/{(2\pi n)}$.}

In particular, the variance of $|a_n| r^n$, for any fixed $r < 1$,
will decay exponentially in $n$.  Thus, the probability that even
one of the $a_n$ satisfies $|a_n| r^n > c$, where $c$ is a fixed
constant, decays quadratic-exponentially in $c$.  It follows that
the probability distribution function $\underline p$ of $\underline h$
satisfies {$\underline p(\eta) :=\mathbb P(\underline h < \eta)
< e^{-\underline C \eta^2}$} for some
$\underline C>0$ and all sufficiently negative $\eta$.

Now, let $P_1(\eta)$ be the probability that $\underline h < .1
\eta/\gamma$ {\em and} $A < \eta$.  Let $P_2(\eta)$ be the
probability that $A < \eta$ and $\underline h \geq .1 \eta/\gamma$.
Then $p_A(\eta) = \mathbb P[A < \eta] = P_1 + P_2$. From the above
discussion, we have $P_1(\eta) \leq e^{-\underline C \eta^2}$ for
all sufficiently negative values of $\eta$. Note from
(\ref{A1A2bound}) that

$$P_2(\eta) \leq \bigl[p_A(.9 \eta + \gamma Q \log 4)\bigr]^2$$
and
$$P_2(\eta) \leq \bigl[ P_1(.9 \eta + \gamma Q \log 4) + P_2(.9 \eta + \gamma Q
\log 4) \bigr]^2 \leq \bigl[e^{-C' \eta^2} + P_2(.9 \eta + \gamma Q
\log 4) \bigr]^2,$$ for some $C'$. Fix a sufficiently negative
$\eta_0$ and inductively determine $\eta_k$ via $\eta_{k-1} = .9
\eta_k + \gamma Q \log 4$. {The above can be stated  as
$$P_2(\eta_k) \leq \left( e^{-C' \eta_{k}^2} + P_2(\eta_{k-1})
\right)^2.$$  If we write $p_k = \frac{P_2(\eta_k)}{e^{-2C'
\eta_k^2}}$, then this can be restated as $p_k \leq
(1+p_{k-1}e^{-C' (2\eta_{k-1}^2 - \eta_k^2)})^2$.  It is easy to
see that we can have $p_k > 2$ for only finitely many $k$,
which
implies that the lemma holds for $C< \min\{\bar C, 2 C'\}$,} when restricted to the sequence
$\eta_k$. Because of the monotonicity of $p_A(\eta)$, this implies
the lemma for all $\eta$. \qed


\begin{lemma}  \label{l.totalmass2}
Fix $z$ and $\varepsilon$ so that $B_\varepsilon(z) \subset D$.
Then
$$\mathbb E[ \mu(B_\varepsilon(z))| h_\varepsilon(z)]  \simeq \pi \varepsilon^{\gamma Q} e^{\gamma
h_\varepsilon(z)} = \mu_{\odot}(B_\varepsilon(z)),$$ where $$Q = \frac{2}{\gamma} +
\frac{\gamma}{2},$$ as in Proposition \ref {Qtransformation}.
Moreover, conditioned on $h_{\varepsilon'}(z)$, for all
$\varepsilon'\geq \varepsilon$, we have that
\begin{equation}\label{massratio}\mathbb P\left[ \frac{\mu(B_\varepsilon(z))}{\mu_\odot(B_\varepsilon(z))} < e^{\eta}\right] \leq C_1 e^{-C_2 \eta^2},\end{equation} for some
positive constants $C_1$ and $C_2$ independent of $\eta\leq 0$, $z$,
$D$, and the values $h_{\varepsilon'}(z)$ for $\varepsilon' \geq
\varepsilon$.
\end{lemma}

\begin{proof}
{The first sentence is a restatement of (\ref{ExpectedQuantumArea}) and (\ref{muodot}).  It remains
to prove the second half.  For a fixed $\varepsilon$, we want to show that the probability that
{$$A := \log \frac{ \mu(B_\varepsilon(z))}{ \pi \varepsilon^{\gamma Q} e^{\gamma
h_\varepsilon(z)}} \leq \eta,\,\,\,\eta \leq 0 ,$$ decays quadratic-exponentially in $\eta$.}
Let $\underline A$ be the infimum of $\gamma\big(\tilde h(\cdot) -h_\varepsilon(z)\big)$
on $B_{\varepsilon/2}(z)$, where $\tilde h$ is $h$ projected onto $\textrm{Harm}\,B_\varepsilon(z)$.} With these definitions, one easily sees that
$$
A' \leq A-\underline A  ,
$$
where $A':=\log \left[\mu_{h-\tilde h}\left(B_{\varepsilon/2}(z)\right)/\pi \varepsilon ^{\gamma Q}\right]$.
In this proof, we let $\mathbb P^{\,\prime}$ denote probability conditioned on $z$ and on the map $\chi:
[\varepsilon,\varepsilon_0^z]\to \mathbb R : \chi({\varepsilon'}) := h_{\varepsilon'}(z)$.

{One may use the techniques in the proof of Lemma \ref{l.totalmass} to show that
{$\mathbb P^{\,\prime}[\underline A <\eta]$} (which is a priori a function of
$\chi$) decays quadratic-exponentially in $\eta$, uniformly in $\chi$ (and hence in $h_\varepsilon(z)$).
{Then we have by the above construction $
\mathbb P^{\,\prime}[A - \underline A < \eta]\leq \mathbb P[A'< \eta]$, so that Lemma \ref{l.totalmass} applied to $p_{A'}(\eta)=
\mathbb P[A'< \eta]$} implies that {$\mathbb P^{\,\prime}[A - \underline A < \eta]$} decays
quadratic-exponentially in $\eta$, also uniformly in $\chi$.  We conclude that the probability that {\it either}
$\underline A <\eta/2$ {\it or} $A-\underline A
< \eta/2$ decays quadratic-exponentially in $\eta$, and the claim follows. \qed}
\end{proof}

Roughly speaking, the above lemma says that the total quantum area
in a ball is unlikely to be a lot smaller than the area we would
predict given the average value of $h$ on the boundary of that ball;
the following says that (even when we use the $\Theta$ measure), the
total quantum area has some constant probability to be (at least a
little bit) smaller than this prediction.

\begin{lemma} \label{l.totalmass3}
Let $z$ and $h$ be chosen from $\Theta^{\tilde D}$ (as defined in Section \ref{s.rootedmetrics}) for a fixed compact subset
$\tilde D$ of $D$, and fix a $\delta
> 0$, with quantum balls $B^{\delta}(z)$ and $\tilde B^{\delta}(z)$
defined as in Definition \ref{Bdelta} and Definition \ref{Btilde}.
(Definition \ref{Btilde} implicitly makes use of a constant $\varepsilon_0$,
which we take here to be $\sup \{ \varepsilon' : B_{\varepsilon'}(\tilde D)
\subset D \}$.)

Conditioned on the radius of $\tilde B^\delta(z)$, the conditional probability
that $\tilde B^\delta(z) \subset B^\delta(z)$ is bounded below by a
positive constant $c$ independent of $D$, $\tilde D$, and $\delta$.
\end{lemma}

\begin{proof}
{
Define $\eps$ to be the radius of $\tilde B^\delta(z)$.
Let $\mathbb P^{\,\prime}$ be $\Theta^{\tilde D}$ probability conditioned on $z$ and on the map
$\chi: [\varepsilon,\varepsilon_0]\to \mathbb R:
\chi(\varepsilon') := h_{\varepsilon'}(z)$.  As before we assume $\varepsilon$ is less than the distance $\varepsilon_0$ from
$\tilde D$ to $\partial D$.  It now suffices for us to show that $$\mathbb P^{\,\prime} \left(\tilde
B^\delta(z) \subset B^\delta(z) \right) = \mathbb P^{\,\prime} \left(\mu(\tilde B^\delta(z))
< \delta\right)$$ is bounded below independently of $\chi$ and $z$.}

{Consider now the map $\overline \chi: (0, \varepsilon_0] \to \mathbb R$ given by $\chi(\varepsilon') = h_{\varepsilon'}(z)$ for all $\varepsilon' \in (0, \varepsilon_0]$.  Let $h^{\overline\chi}$ denote the conditional expectation of $h$ (in the standard GFF probability measure) given $\overline \chi$.  Clearly $h^{\overline\chi}(y)$ is a.s.\ radially symmetric, with center $z$, and equals zero for $ \varepsilon_0 <|y-z| $.  It corresponds to the projection of $h$ onto the space of functions with these properties.  (We similarly define $h^\chi$, so that $h^\chi$ is constant in $B_\varepsilon(z)$ and coincides with $h^{\overline \chi}$ outside.)  Note that, given $z$, the $\Theta^{\tilde D}$ law of $h^{\overline\chi}$ is that of $\tilde h^{\overline\chi}$ (where $\tilde h$ is a standard GFF) plus a deterministic function with the same properties: the function $$\zeta(y) := \gamma[\xi_0^z(y)-\xi_{\varepsilon_0}^z(y)] = \begin{cases} -\gamma \log \frac{|y-z|}{\varepsilon_0} & y \in B_{\varepsilon_0}(z) \\ 0 & \mathrm{otherwise}\end{cases}.$$
Although $h^{\overline\chi}$ is a projection onto an infinite dimensional space, it is not hard to see (e.g., by approximating with finite dimensional spaces) that the obvious analog of \eqref{condexpect} in Proposition \ref{p.hnlimit} still holds, i.e., taking expectation with respect to $\Theta^{\tilde D}$ we have
$$\mathbb E [ \mu(A) | \overline\chi] = \mu^{\overline\chi}(A),$$
where $$\mu^{\overline\chi}:=  \exp \left( \gamma h^{\overline\chi}(y) +
\frac{\gamma^2}{2}\log\frac{|y-z|}{\varepsilon_0} + \frac{\gamma^2}{2}\log C(y; D)
\right)dy,$$
for $|y-z| \leq \varepsilon_0$, and in this range we have for some $C_0\geq1$ (depending on $\tilde D$ and $D$) that
\begin{equation}\label{Cbound} C_0^{-1}\mu^{\overline\chi}_0 \leq  \mu^{\overline\chi} \leq C_0\,\mu^{\overline\chi}_0,\end{equation} where
\begin{equation}\label{muchi}\mu^{\overline\chi}_0:=  \exp \left( \gamma h^{\overline\chi}(y) +
\frac{\gamma^2}{2}\log\frac{|y-z|}{\varepsilon_0}\right)dy.\end{equation}
The fact that $\mu^{\overline\chi}$ and hence $\mu^{\overline\chi}_0$ is almost surely finite follows from the fact that it is a conditional expectation of $\mu=\mu_h$ with
respect to $\Theta^{\tilde D}$, and $\mu$ is $\Theta^{\tilde D}$ almost surely finite.  It can also be seen directly from \eqref{muchi}, using
the same argument as in the proof of Proposition \ref{p.hnlimit}: note first that \eqref{muchi} becomes integrable for $\gamma \in [0,2)$ if $h^{\overline \chi}(y)$ is replaced with its expectation $\zeta(y)$, and second that $|h^{\overline \chi}(y) - \zeta(y)|$ a.s.\ does not grow too quickly as $y \to z$.}

{Note that $h_\varepsilon^{\overline \chi}(z) = h_\varepsilon^{\chi}(z) = h^\chi(z) = h_{\varepsilon}(z) - h_{\varepsilon_0}(z)$.
From definition 
 \eqref{muodot} and from \eqref{mudotmu0dot}
 and the definition (\ref{Btilde}) of $\tilde B^\delta(z)=B_\varepsilon(z)$,
we have that $$\left(\frac{\varepsilon}{\varepsilon_0}\right)^2 \exp \left( \gamma h^\chi_\varepsilon(z) + \frac{\gamma^2}{2}\log \frac{\varepsilon}{\varepsilon_0}\right)= \delta$$
and that
\begin{equation}\label{deltaratio}
\delta^{-1}\mu_0^{\bar \chi}(\tilde B^{\delta}(z))
\end{equation}
is a random variable independent of
$\varepsilon$ and $\chi$.  (It depends only on the Brownian process given by $\tilde {\mathcal B}_{s} := {\mathcal B}_{s+t} - {\mathcal B}_t$ defined for $s \geq 0$, where ${\mathcal B}_{s'}:= h^{\overline \chi}_{e^{-s'}}(z) - \gamma s'$ and $t = -\log (\varepsilon/\varepsilon_0)$.  Note that $\tilde {\mathcal B}_s$ is a standard Brownian motion in $s \geq 0$, independent of $\varepsilon$, $z$, and $\chi$.)}

{It is not hard to see that this random variable is not bounded below by any number greater than zero; thus there is an event --- call it $\mathcal A$ --- independent of $\chi$, and occurring with a probability  
$\mathbb P^{\,\prime}(\mathcal A)$ bounded below by some $c'>0$, on which \eqref{deltaratio} is less than a small number, say $1/100$ (indeed, we may assume the same holds with $\mu^{\overline \chi}$ replacing $\mu^{\overline\chi}_0$, because of \eqref{Cbound}).
Given this claim, it follows that on the event $\mathcal A$, one has $\mathbb E [ \mu\big(\tilde B^\delta(z)\big) | \overline\chi] = \mu^{\overline\chi}\big(\tilde B^\delta(z)\big)<\delta/100$, so that} the conditional probability that $\mu(\tilde
B^\delta(z)) < \delta$ is at least $.99$.
 The lemma follows using the constant $c=.99 c'$. \qed
\end{proof}

\subsection{Proof of interior KPZ}
In this section we derive Theorem \ref{t.verystrongquantumKPZ} as a
consequence of Lemma \ref{l.totalmass2}, Lemma \ref{l.totalmass3},
and the arguments in Theorem \ref{t.expectationquantumKPZ}.

\proofof{Theorem \ref{t.verystrongquantumKPZ}} We use the same
notation as in Theorem \ref{t.expectationquantumKPZ}, but we write
$\overline T_A = - \log (\overline \varepsilon/\varepsilon_0)$ where $\overline
\varepsilon$ is the radius of $B^\delta(z)$.  In this proof, we use the probability measure
$\Theta^{\tilde D}$ and $\mathbb E$ denotes expectation with respect to $\Theta^{\tilde D}$.
The proof of
Theorem \ref{t.expectationquantumKPZ} carries through exactly once
we show that (when $h$ is chosen from $\Theta^{\tilde D}$) \begin{equation}\label{e.ATAlimit} \lim_{A \to \infty}
\frac{\log \mathbb E \left[\exp{(-2x\overline T_A)}\right]}
{\log\mathbb E \left[\exp{(-2x T_A)}\right]} = 1,\end{equation}
since this implies the analog of (\ref{e.expmartingalekpz}) with $T_A$ replaced by $\overline T_A$.

Note that the numerator of (\ref{massratio}) is related to $\overline T_A$ while the denominator
is related to $T_A$; if the numerator and denominator were precisely equal for all $\varepsilon$,
we would have $\overline T_A = T_A$.

{For any $a$,
$0 < a < 1$, let $\varepsilon_a$ be the value for which $B_{\varepsilon_a}(z)=\tilde B^{\delta^a}(z)$.  Then
 $\mu_\odot\left( B_{\varepsilon_a}(z)\right)
=\delta^a\mu_\odot\left( B_{\varepsilon_0}(z)\right)$. This corresponds to a stopping time
$T_{a A}=-\log(\varepsilon_a/\varepsilon_0)$.}

{On the event  $\overline T_{A} < T_{aA}$, we have
$\varepsilon_a < \overline \varepsilon$ so that $\mu\left(B_{ \varepsilon_a}(z)\right)
\leq \mu\left(B_{\overline \varepsilon}(z)\right)=\delta $. It follows that
$$
\mu\left(B_{ \varepsilon_a}(z)\right)/\mu_\odot\left( B_{\varepsilon_a}(z)\right)\leq
\delta^{1-a}/\mu_\odot\left( B_{\varepsilon_0}(z)\right).
$$ Thanks to definition (\ref{muodot}), the probability that $\mu_\odot\left( B_{\varepsilon_0}(z)\right)\leq \delta^{(1-a)/2}$
decays quadratic-exponentially in $A=-\log (\delta/\gamma)$ when $\delta \to 0$. On the event of the contrary,
 $\mu_\odot\left( B_{\varepsilon_0}(z)\right)>\delta^{(1-a)/2}$, one then has
$
\mu\left(B_{ \varepsilon_a}(z)\right)/\mu_\odot\left( B_{\varepsilon_a}(z)\right)<
\delta^{(1-a)/2},
$
whose probability, by Lemma \ref{l.totalmass2} applied for $B_{\varepsilon_a}(z)$ and  $\eta=-\gamma A(1-a)/2$, also decays
quadratic-exponentially in $A$. This implies that the probability that
$\overline T_{A} < T_{aA}$ decays superexponentially in $A$.} This implies that
$$\lim_{A \to \infty} \frac{\log \mathbb E \left[\exp{(-2x\overline T_{A})}\right]}
{\log\mathbb E \left[\exp{(-2x T_{aA})}\right]} \leq 1.$$  {Since
this holds for all $a<1$, it follows immediately from the continuity of the coefficient of $A$ in
the exponent in (\ref{e.expmartingalekpz}) that
$$\lim_{A \to \infty} \frac{\log \mathbb E \left[\exp{(-2x\overline T_{A})}\right]}
{\log\mathbb E \left[\exp{(-2x T_{A})}\right]} \leq 1.$$
From Lemma
\ref{l.totalmass3}, it follows that, conditioned on $T_A$, the $\Theta^{\tilde D}$ probability that $\overline T_A < T_A$ is at least $c >0$, which implies
$$c\, \mathbb E \left[\exp{(-2xT_{A})}\right] \leq \mathbb E \left[\exp{(-2x \overline T_{A})}\right]$$
 for any $x \geq 0$, which in turn implies the equivalence of logarithms in  (\ref{e.ATAlimit}). \qed
}

\section{Box formulation of KPZ}
In this section we prove Proposition \ref{p.boxscalingdimension}.

\proofof{Proposition \ref{p.boxscalingdimension}}  The first fact is
standard; simply observe that if $\varepsilon$ is a power of $2$ then
$S_{\varepsilon}(X) \subset B_{2\varepsilon}(X)$, {hence $\mu_0(S_{\varepsilon}(X)) \leq
\mu_0(B_{2\varepsilon}(X))$,} since the ball of
radius $2 \varepsilon$ about a point contains any diadic box of
width $\varepsilon$ that contains the same point. Similarly,
$B_{2\varepsilon}(z)$ is contained in the union of a diadic box
--- of width $2\varepsilon$, containing $z$ --- with the eight diadic
boxes of the same size whose boundaries touch its boundary. This
implies that $B_{2\varepsilon}(X)$ is contained in the union of
$S_{2\varepsilon}(X)$ and corresponding $8$ translations of
$S_{2\varepsilon}(X)$, so $\mu_0(B_{2\varepsilon}(X)) \leq 9
\mu_0(S_{2\varepsilon}(X))$.

For the second part, we first argue that $X$ has quantum scaling exponent $\Delta$ if
and only if \eqref{e.qse} holds.  We use the notation
in the proof Theorem \ref{t.expectationquantumKPZ} but set $\tilde T_A$ to be
 {$- \log (\tilde \varepsilon/\varepsilon_0)$, where $\tilde \varepsilon$ is the largest value of $\varepsilon$ for
which the diadic box $S_\varepsilon(z)$} with edge length $\varepsilon$ has $\mu$ area
at most $\delta$.  The remainder of the argument is essentially the
same as the proof of Theorem \ref{t.verystrongquantumKPZ}.
Just as \eqref{e.ATAlimit} was sufficient in that case, it is enough for
us to verify that when $h$ is chosen from $\Theta^{\tilde D}$ and $X$ is chosen independently,
we have the following analog of \eqref{e.ATAlimit}
 {(where $\overline T_A$ is replaced by $\tilde T_A$):
\begin{equation}\label{e.ATAlimit2} \lim_{A \to \infty}
\frac{\log \mathbb E \left[\exp{(-2x\tilde T_A)}\right]}
{\log\mathbb E \left[\exp{(-2x  T_A)}\right]} = 1.\end{equation}}

The proof is essentially the same as the proof of (\ref{e.ATAlimit}), but we will sketch the differences here.
As in the proof of \eqref{e.ATAlimit} one argues first that the probability that
 {$\tilde T_{A} <  T_{aA}$, with $0<a<1$,} decays superexponentially in $A$ and
 {by continuity when $a \to 1$} concludes that
 {$$\lim_{A \to \infty} \frac{\log \mathbb E \left[\exp{(-2x \tilde T_{A})}\right]}
{\log\mathbb E \left[\exp{(-2x  T_{A})}\right]} \leq 1.$$}  The only difference is that
one has to first obtain a modified Lemma \ref{l.totalmass2}, in which the $\mu(B_\varepsilon(z))$ in \eqref{massratio} is
replaced with $\mu(S_{\varepsilon/2}(z))$; this straightforward exercise is left to the reader.
 {Then, using the same notation as in the proof of Theorem \eqref{t.verystrongquantumKPZ}, one
chooses $\varepsilon_a$ so that $ B_{\varepsilon_a}(z) = \tilde B^{\delta^a}(z)$.  Then $\mu_\odot\left( B_{\varepsilon_a}(z)\right)=
\delta^{a}\mu_\odot\left( B_{\varepsilon_0}(z)\right)$. On
   the event $\tilde T_{A} <  T_{aA}$, one has $\varepsilon_a < \tilde \varepsilon$, so that
  $\mu(S_{\varepsilon_a/2}(z)) \leq  \mu(S_{\varepsilon_a}(z))\leq \mu(S_{\tilde\varepsilon}(z))<\delta$,
  from which it follows that
  $$
\mu\left(S_{ \varepsilon_a/2}(z)\right)/\mu_\odot\left( B_{\varepsilon_a}(z)\right)\leq
\delta^{1-a}/\mu_\odot\left( B_{\varepsilon_0}(z)\right).
$$
The discussion then continues identically, depending on whether $\mu_\odot\left( B_{\varepsilon_0}(z)\right)
\leq \delta^{(1-a)/2}$ holds, the probability of which has superexponential decay in $A=-(\log \delta )/\gamma$,
or the contrary, which also has superexponential decay by application of the modified Lemma \ref{l.totalmass2} to
 the resulting inequality
  $$
\mu\left(S_{ \varepsilon_a/2}(z)\right)/\mu_\odot\left( B_{\varepsilon_a}(z)\right)\leq
\delta^{(1-a)/2}.
$$
}

Next, as in the proof of \eqref{e.ATAlimit}, one argues that
 {$\mathbb P^{\,\prime}[\tilde T_A <  T_A +\log \rho] \geq c >0$, for some
fixed constant $\rho \geq 4$, which implies
$$c\,\rho^{-2x}\, \mathbb E \left[\exp{(-2xT_{A})}\right] \leq \mathbb E \left[\exp{(-2x \tilde T_{A})}\right]$$
for any $x \geq 0$, which in turn implies (\ref{e.ATAlimit2}).}  The difference here is that one must first
obtain a modified version of
Lemma \ref{l.totalmass3} in which the event $\tilde B^\delta(z) \subset B^\delta(z)$ is replaced
with the event that {$S^\delta(z)=S_{\tilde \varepsilon}(z)$ has a  width $\tilde \varepsilon$
larger than a fixed constant $\rho^{-1}$
times the radius $\varepsilon$ of $\tilde B^\delta(z)=B_\varepsilon(z)$, which can be easily proven as follows.
First, recall that by definition of $\tilde \varepsilon$,
$\mu\left(S_{\tilde \varepsilon}(z)\right) <\delta \leq \mu\left(S_{2\tilde \varepsilon}(z)\right)$.
We thus have
$$
\mu\left(B_{\overline \varepsilon}(z)\right)=\delta\leq \mu\left(S_{2\tilde \varepsilon}(z)\right)\leq
\mu\left(B_{4\tilde \varepsilon}(z)\right),
$$
hence $\overline \varepsilon\leq 4\tilde \varepsilon$. From Lemma \ref{l.totalmass3} there is a finite probability $c$ that $\tilde B^\delta(z) \subset B^\delta(z)$,
i.e., that $\varepsilon \leq \overline \varepsilon$, hence that $\varepsilon \leq 4\tilde \varepsilon$, which proves
the modified version of Lemma \ref{l.totalmass3} for $\rho \geq 4$.}

Next, we observe that
the above arguments still work if we replace the $S^\delta(z)$ in \eqref{e.qse}
with $\hat S^\delta(z)$, defined to be the diadic parent of $S^\delta(z)$ --- this only changes $\tilde T_A$ by
an additive constant.  Thus \eqref{e.qse} is equivalent
to the analog of \eqref{e.qse} in which $S^\delta(z)$ is replaced with $\hat S^\delta(z)$.
Now define $\hat N$ analogously to $N$ (counting $\hat S^\delta(z)$ squares instead of $S^\delta(z)$ squares).  We obtain
the equivalence of  \eqref{e.qse} and \eqref{e.ne} by observing that
\begin{eqnarray*}\lim_{\delta \to 0} \frac{\log \mathbb E \left[\mu(S^\delta(X))\right]}{\log \delta} & \leq & \lim_{\delta \to 0} \frac{ \log\mathbb E \left[\delta N(\mu, \delta,
X)\right]}{\log \delta} \\ & \leq & \lim_{\delta \to 0} \frac{\log\mathbb E \left[\delta \hat N(\mu, \delta, X)\right]}{\log \delta} \\ & \leq & \lim_{\delta \to 0} \frac{\log \mathbb E \left[\mu(\hat S^\delta(X))\right]}{\log \delta}.\end{eqnarray*}
The first and last inequalities are true because, by definition, $\mu(S^\delta(z)) \leq \delta$ and $\mu(\hat S^\delta(z)) \geq \delta$.
The middle inequality is true because $N(\mu, \delta, X) \leq 4 \hat N(\mu, \delta, X)$.
\qed

\section{Boundary KPZ} \label{boundaryKPZsection}
\subsection{Boundary semi-circle average}
Most of the results in this paper about random measures on $D$ have
straightforward analogs about random measures on $\partial D$.  The
proofs are essentially identical, but we will sketch the differences
in the arguments here.

Suppose that $D$ is a domain  {with piecewise linear boundary or a domain with a smooth boundary containing  a linear piece $\underline{\partial D}\subset \partial D$} and that
$h$ is an instance of the GFF on $D$ with free boundary conditions,
normalized to have mean zero.

This means that $h = \sum_n \alpha_n f_n$ where the $\alpha_n$ are
i.i.d.\ zero mean unit variance normal random variables and the
$f_n$ are an orthonormal basis, with respect to the inner product
$$(f_1, f_2)_\nabla := (2\pi)^{-1} \int_D \nabla f_1(z) \cdot \nabla f_2(z) dz,$$
of the Hilbert space closure $H(D)$ of the space of $C^\infty$
bounded real-valued (but {\em not} necessarily compactly supported)
functions on $D$ with mean zero.

Note that if $f$ is a compactly supported smooth function on $D$ for
which $-\Delta f = \rho$, then integration by parts implies that the
variance of $(h, \rho)$ is the Dirichlet energy of $f$---same as in
the zero boundary case.  Similarly, suppose that $f$ is a smooth
function that is {\em not} compactly supported but has a gradient
that vanishes in the normal direction to $\partial D$, and we write
$\rho = - \Delta f$.  Then integration by parts implies that the
variance of $(h, \rho)$ is $2\pi(f,f)_\nabla$.

We can also make sense of $h_\varepsilon(z)$, for a point $z$ on a
linear part $\underline{\partial D}$ of the boundary of $D$, to be the mean value of $h$ on the semicircle of
radius $\varepsilon$ centered at $z$ and contained in the domain
$D$. {For $z$ fixed, let $\varepsilon_0$ be chosen small enough so that
$B_{\varepsilon_0}(z)\cap \partial D \subset \underline{\partial D}$  and exactly one
semi-disc of $B_{\varepsilon_0}(z)$ lies in $D$. Define for any $\varepsilon \leq\varepsilon_0$
$$
h_\varepsilon(z)=(h,\hat \rho_\varepsilon^z),
$$
where $\hat\rho_\varepsilon^z(y) dy$ is the uniform measure (of total mass one)  localized on the semicircle
$\partial B_\varepsilon(z)\cap D$.
Let us introduce the function $\hat \xi_\varepsilon^z(y)$, for $y\in D$, such that
\begin{eqnarray}\label{c1}
-\Delta\hat \xi_\varepsilon^z&=&2\pi \left(\hat\rho_\varepsilon^z-{1}/{|D|}\right),\\
\label{c3}
n\cdot\nabla\hat\xi_\varepsilon^z|_{\partial D}=0,&&
\int_D \hat \xi_\varepsilon^z dy =0,
\end{eqnarray}
with $n$ the current normal to $\partial D$, and  $|D|:=\int_D dy$ the area of $D$. Hence,
$\hat \xi_\varepsilon^z$ satisfies \textit{Neumann} boundary conditions and has zero mean, and integration by
parts shows  that
$$
h_\varepsilon(z)=\big(h,\hat \xi_\varepsilon^z\big)_\nabla.
$$
Let us introduce the auxiliary function
\begin{eqnarray}\label{zetaeps}
\zeta_\varepsilon^z(y):=-2\log (|y-z|\vee \varepsilon)+\frac{\pi}{2|D|}\left(|y-z|^2+\varepsilon^2 \right),
\end{eqnarray}
such that $-\Delta\zeta_\varepsilon^z(\cdot)=2\pi \left(\hat\rho_\varepsilon^z(\cdot)-{1}/{|D|}\right)$.
The $2\log (|\cdot - z|\vee \varepsilon)$ in place of $\log( |
\cdot - z|\vee \varepsilon)$ comes from the fact that $\hat \rho_\varepsilon^z$ is a unit mass measure over half
a circle.}

{The solution  $\hat \xi_\varepsilon^z$ to 
 (\ref{c1}) and (\ref{c3}) is then given by
\begin{eqnarray}\label{xizetaG}
\hat \xi_\varepsilon^z&=&\zeta_\varepsilon^z-\hat G_z,
\end{eqnarray}
where $\hat G_z$ is the {harmonic} function in $D$, solution to the Neumann problem \eqref{c3} on $\partial D$.
Note that the function $\zeta_\varepsilon^z$  \eqref{zetaeps} has been chosen such that both the boundary normal derivative $n\cdot\nabla\zeta_\varepsilon^z|_{\partial D}$ and the integral $\int_D\zeta_\varepsilon^z dy$ are actually \textit{independent} of  $\varepsilon$ for $\varepsilon \leq \varepsilon_0$.
The normal derivative vanishes on the linear boundary component $\underline{\partial D}$:  $n\cdot\nabla\zeta_\varepsilon^z|_{\underline{\partial D}}=0$.
Therefore $\hat G_z$ is independent of $\varepsilon$,
and  
satisfies the
Neumann condition on $\underline{\partial D}$:
$n\cdot\nabla\hat G_z|_{\underline{\partial D}}=0.$
By the Schwarz reflection principle, this allows extending
$\hat G_z$ to a harmonic function in the domain $\bar D$, complex conjugate and symmetrical of $D$ with respect to
$\underline{\partial D} \subset \mathbb R$,
through  $\hat G_z(\bar y)=\hat G_z(y)$.}

{When considering the reference radius $\varepsilon_0$, we then have that
\begin{eqnarray}\label{hath-h0}
h_{\varepsilon}(z) - h_{\varepsilon_0}(z)=\big(h,\hat\xi_\varepsilon^z-\hat\xi_{\varepsilon_0}^z\big)_\nabla
=\big(h,\zeta_\varepsilon^z-\zeta_{\varepsilon_0}^z\big)_\nabla .
\end{eqnarray}
Thus $h_{\varepsilon}(z) - h_{\varepsilon_0}(z)$ is equal to
$(h,\hat\zeta)_\nabla$, where $\hat \zeta:=\zeta_\varepsilon^z-\zeta_{\varepsilon_0}^z$ (up to a constant) is
 the continuous function
to $-2\log |\cdot - z |$ on the half-annulus $\H \cap
\{y: \varepsilon \leq |y-z| \leq \varepsilon_0 \}$ and is constant outside
of the half-annulus. The variance of $h_{\varepsilon}(z) -
h_{\varepsilon_0}(z)$ is then given by the Dirichlet energy
$(\hat\zeta,\hat\zeta)_\nabla=- 2\log (\varepsilon/\varepsilon_0)$. We thus have that the Gaussian random variable
$h_{\varepsilon}(z) - h_{\varepsilon_0}(z)$  is a standard Brownian
motion $\mathcal B_{2t}$ in time $2t = - 2\log (\varepsilon/\varepsilon_0)$, with boundary condition $\mathcal B_0=0$, as in
Proposition \ref{p.browniancircleaverages}.}


{Thanks to equations (\ref{c1}) to (\ref{xizetaG}),
the set of functions $\hat \xi_\varepsilon^z$
has Dirichlet inner products
\begin{eqnarray}\label{hatxixi}
\big(\hat \xi_\varepsilon^z,\hat \xi_{\varepsilon'}^z\big)_\nabla=-2\log (\varepsilon\vee \varepsilon')
+\frac{\pi}{2|D|}\big(\varepsilon^2+\varepsilon'^2\big)-\hat G_z(z);
\end{eqnarray}
one finds in particular for $\varepsilon'=0$  that $ \big(\hat \xi_\varepsilon^z,\hat \xi_{0}^z\big)_\nabla=\hat \xi_\varepsilon^z(z)$.
At a boundary point $z\in \underline{\partial D}$, the variance of $h_\varepsilon(z)$ is
\begin{eqnarray}\label{Varhat}
\Var\, h_\varepsilon(z)=
\big(\hat \xi_\varepsilon^z,\hat \xi_{\varepsilon}^z\big)_\nabla=-2\log \varepsilon
+\frac{\pi}{|D|}\varepsilon^2-\hat G_z(z);
\end{eqnarray}
this variance thus scales for $\varepsilon$ small like $-2\log
\varepsilon$ instead of $-\log \varepsilon$, because of the free boundary conditions on $\underline{\partial D}$.}
{\subsection{Mixed boundary conditions}
Notice that one can also consider other types of boundary conditions for the Gaussian free field $h$, like
\textit{mixed} boundary conditions in domain $D$, with \textit{free} boundary conditions on a  linear component $\underline{\partial D}  \subset \mathbb R$, and \textit{Dirichlet} ones on its complement
 $\partial D\setminus \underline{\partial D}$. In this case, one uses a reflection principle and considers
    the whole domain $D^{\dagger}:=D\cup\bar D$, where $\bar D$ is the complex conjugate of $D$, symmetrical of $D$ with respect to the real axis, and takes Dirichlet boundary conditions on $\partial D^{\dagger}$. The Hilbert space closure $H(D^{\dagger})$ of the space of $C^\infty$
 real-valued  functions compactly supported on $D^{\dagger}$ can be written as the direct sum $H_{\rm e}(D^{\dagger})\oplus H_{\rm o}(D^{\dagger})$ of the Hilbert space closures corresponding to \textit{even} and \textit{odd} functions on $D^{\dagger}$ with respect to the real line supporting $\underline{\partial D}$. The Gaussian free field $h$ in $D$, with mixed boundary conditions on $\partial D$, is then simply obtained by projecting the GFF in $D^{\dagger}$ onto the even space $H_{\rm e}(D^{\dagger})$, and restricting the result to $D$.}

{It is not hard to see that the semi-circle average $h_\varepsilon(z)$ of $h$, for $z\in \underline{\partial D}$ and $\varepsilon\leq \varepsilon_0$, is then given by $\big(h,\tilde\xi_\varepsilon^z\big)_\nabla$, where $\tilde \xi_\varepsilon^z(y)=-2\log \big(|y-z|\vee \varepsilon\big)-\tilde G_z(y)$, with now
  $\tilde G_z(y)$  the harmonic extension to $D^{\dagger}$ (here restricted to $y\in D$) of the restriction of the function $-2\log |y-z|$  to $y\in \partial D^{\dagger}$.  These functions have Dirichlet inner products
  \begin{eqnarray}\label{tildexixi}
\big(\tilde \xi_\varepsilon^z,\tilde \xi_{\varepsilon'}^z\big)_\nabla=-2\log (\varepsilon\vee \varepsilon')
-\tilde G_z(z),
\end{eqnarray}
in place of \eqref{hatxixi}. Similarly, at a boundary point $z\in \underline{\partial D}$, the variance of $h_\varepsilon(z)$ is
\begin{eqnarray}\label{Vartilde}
\Var\, h_\varepsilon(z)=
\big(\tilde \xi_\varepsilon^z,\tilde \xi_{\varepsilon}^z\big)_\nabla=-2\log \varepsilon
-\tilde G_z(z),
\end{eqnarray}
instead of \eqref{Varhat}.
Lastly, exactly as in the case of free boundary conditions, the Gaussian random variable $h_\varepsilon(z)-h_{\varepsilon_0}(z)$ has variance $-2\log (\varepsilon/\varepsilon_0)$, and is a standard Brownian
motion $\mathcal B_{2t}$ in time $2t = - 2\log (\varepsilon/\varepsilon_0)$, with initial value $\mathcal B_0=0$.}

{In the following section, we shall consider equally well free or mixed boundary conditions, up to some minor differences that are mentioned in each case.}

\subsection{Boundary measure and KPZ}
We define the boundary
measure $\mu^B_\varepsilon := \varepsilon^{\gamma^2/4}e^{\gamma
h_\varepsilon(z)/2}dz$, where in this case $dz$ is Lebesgue measure
on the boundary component $\underline{\partial D}$.  Here we use $e^{\gamma h_\varepsilon(z)/2}$
instead of $e^{\gamma h_\varepsilon(z)}$ because we are integrating
a length instead of an area; as before, the power of $\varepsilon$
that we chose makes the expectation of the factor preceding $dz$, $\varepsilon^{\gamma^2/4} \mathbb E e^{\gamma
h_\varepsilon(z)/2}=\varepsilon^{\gamma^2/4}e^{\gamma^2
\Var \,h_\varepsilon(z)/8}$, have a finite limit when $\varepsilon \to 0$.

We define $\mu^B$ to be the weak limit as $\varepsilon \to 0$ of the
measures $\mu^B_\varepsilon$ (see the theorem below for existence of
this limit when $0 \leq \gamma < 2$). For $z\in
\underline{\partial D}$ we write $\hat B_\varepsilon(z) := B_\varepsilon(z) \cap
\underline{\partial D}$ and we define $\hat B^\delta(z)$ to be the (largest) set $\hat
B_\varepsilon(z)$ whose $\mu^B$ measure is $\delta$.

Likewise define $$\hat B_\varepsilon (X) = \{z \in
\underline{\partial D}: \hat B_\varepsilon (z) \cap X \not = \emptyset \}$$ and
$$\hat B^\delta(X) = \{z \in \underline{\partial D}: \hat B^\delta(z) \cap X \not = \emptyset \}.$$
We say that a (deterministic or random) fractal subset $X$ of the
boundary  component $\underline{\partial D}$ has {\bf Euclidean expectation dimension} $1-\tilde
x$ and {\bf Euclidean scaling exponent $\tilde x$} in the boundary
sense if the expected measure of $\hat B_\varepsilon(X)$ decays like
$\varepsilon^{\tilde x}$, i.e.,
$$\lim_{\varepsilon \to 0} \frac{ \log \mathbb E \mu_0(\hat B_\varepsilon(X)) }{\log \varepsilon} = \tilde  x.$$
We say that $X$ has {\bf boundary quantum scaling exponent $\tilde
\Delta$} if when $X$ and $\mu^B$ (as defined above) are chosen
independently we have
$$\lim_{\delta \to 0} \frac{ \log \mathbb E \mu^B(\hat B^\delta(X)) }{\log \delta} = \tilde \Delta.$$
\begin{theorem}
Given the assumptions above, Proposition \ref{p.hepsilonlimit} and
Theorems \ref{t.verystrongquantumKPZ} and \ref{t.expectationquantumKPZ} hold, precisely as stated,
when $\mu_\varepsilon$ is replaced by $\mu^B_\varepsilon$, $\mu$ is replaced by $\mu^B$; $\mu_0$
(Lebesgue measure on $D$) is replaced by Lebesgue measure on
 one of the
boundary line segments $\underline{\partial D}$ of $D$; $B_\varepsilon$ and $B^\delta$ are replaced with $\hat
B_\varepsilon$ and $\hat B^\delta$, respectively; and the compact
subset of $D$ is replaced with a closed subinterval of $\underline{\partial D}$.
\end{theorem}

\begin{proof}
The proofs in the boundary case proceed exactly the same as in the
interior point case, up to factors of $2$ in various places. We
sketch the proof of an analog of Theorem
\ref{t.expectationquantumKPZ} in order to indicate where those
factors of $2$ appear.

Write $t := - \log (\varepsilon/\varepsilon_0)$, and let $V_t :=
h_\varepsilon(z)-h_{\varepsilon_0}(z)$. It is not hard to see that
the expectation of the boundary line integral
$$\mathbb E_h \left[\int_{\hat B_\varepsilon(z)} e^{\gamma h/2}dy |
V_t\right]=\mathbb E_h \left[\mu^B_h({\hat B_\varepsilon(z)})|h_\varepsilon(z)-h_{\varepsilon_0}(z)\right]
$$ has approximately the form (which replaces (\ref{muBepsbis}) and (\ref{mu1BepsilonVt}))
\begin{equation} \label{ExpectedQuantumAreabis}
\exp\left(\frac{\gamma}{2} V_t- \frac{\gamma}{2} Q t\right),
\end{equation}
in the sense that the ratio of the logarithms of the two quantities tends
to $1$ when $\varepsilon \to 0$ and $t \to \infty$.  Let ${\tilde B}^\delta(z)$ now be the largest Euclidean
ball $B_{\varepsilon}(z)$ in $D$ centered at $z\in
\underline{\partial D}$ for which (\ref{ExpectedQuantumAreabis}) is equal to the
quantum length $\delta$, and ${\tilde B}^\delta(X) := \{z \in \underline{\partial D} : \tilde B^\delta(z) \cap X \not = \emptyset \}.$
\\

{As before, we use the fact that $\mathbb E_h \mu_h^B\big({\tilde B}^\delta(X)\big)$ is proportional to $\hat\Theta\{(z,h): z \in \underline{\partial D} ,  \tilde B^\delta(z) \cap X \not = \emptyset\}$, where $\hat \Theta$ is the  boundary rooted measure such that, given {$z \in \underline{\partial D}$},  $h$ is sampled from the Gaussian free field distribution
\textit{weighted} by $e^{\gamma h(z)/2}$.  For \textit{free} boundary conditions, the $\hat \Theta$ conditional law of $h$ is then that of
the original GFF \textit{plus} the deterministic function $\frac{\gamma}{2} \hat \xi_0^z(\cdot)=-\gamma\log |\cdot-z|+\frac{\gamma\pi}{4|D|}|\cdot-z|^2$.
Then given $z\in \underline{\partial D}$, the $\hat \Theta$ conditional law of
the semi-circular average $h_\varepsilon(z)=\big(h,\hat \xi_\varepsilon^z\big)_\nabla$ is
that of the original GFF semi-circular average,  plus the Dirichlet inner product
$\frac{\gamma}{2}\big(\hat \xi_0^z,\hat \xi_\varepsilon^z\big)_\nabla=\frac{\gamma}{2}\hat \xi_\varepsilon^z(z)$.}

{Then given $z\in \underline{\partial D}$, the $\hat\Theta$ conditional law of
$ V_t=h_\varepsilon(z)-h_{\varepsilon_0}(z)$
is that of
$${\mathcal B}_{2t} + \frac{\gamma}{2}\big(\hat \xi_\varepsilon^z(z)-\hat \xi_{\varepsilon_0}^z(z)\big)
={\mathcal B}_{2t} -\gamma\log(\varepsilon/\varepsilon_0)+b_{\varepsilon}-b_{\varepsilon_0}={\mathcal B}_{2t} +\gamma t+b_{\varepsilon_0e^{-t}}-b_{\varepsilon_0},$$
 where $b_{\varepsilon}:=\frac{\gamma}{2}\frac{\pi\varepsilon^2}{2|D|}$; thus  $V_t$ evolves \textit{independently} of $z$, as a Brownian motion ${\mathcal B}_{2t}$ with {twice} the variance of standard Brownian motion, because of the free boundary conditions on
$\underline{\partial D}$, plus a drift term $\gamma t$, and up to a constant and an exponentially small term when $t\to \infty$.}

{In the case of \textit{mixed} boundary conditions, the same line of arguments (recall \eqref{tildexixi}) shows that the $\hat\Theta$ conditional law of
$ V_t=h_\varepsilon(z)-h_{\varepsilon_0}(z)$
is simply that of
$${\mathcal B}_{2t} + \frac{\gamma}{2}\big(\tilde \xi_\varepsilon^z(z)-\tilde \xi_{\varepsilon_0}^z(z)\big)
={\mathcal B}_{2t} +\gamma t.$$}

Using (\ref{ExpectedQuantumAreabis}), we have {both for free and mixed boundary conditions}
\begin{eqnarray}
\label{exphBis} \mathbb E \left[\int_{\hat B_\varepsilon(z)}
e^{\gamma h/2}dy |V_t\right] \asymp
\exp\left(\frac{\gamma}{2} {\mathcal B}_{2t} +\frac{1}{2} \gamma^2 t
- \frac{\gamma}{2} Q t\right).
\end{eqnarray}
This will be equal to the quantum boundary length $\delta$ at the
smallest $t$ for which
$\gamma {\mathcal B}_{2t} +\gamma^2 t  -
\gamma Q t = 2\log \delta$, with ${\mathcal B}_{0}=0$.
If we set $A := -(\log \delta) /
\gamma$, this smallest time is a stopping time $T_A$
such that
\begin{equation}
\label{stoppingtimebis} T_A=\inf\{t: {\mathcal B}_{2t} +
at=2A=-2(\log \delta)/\gamma\} ,\,\, a=Q-\gamma
=\frac{2}{\gamma}-\frac{\gamma}{2}> 0.
\end{equation}
As above, we consider the two part experiment in which we first
sample $T_A$ and then sample $z$ and check to see whether the ball
of radius $\varepsilon=\varepsilon_0\,  e^{-T_A}$ intersects $X$ on the boundary.
Given $T_A$, the ratio of the logarithms of this probability and
$$\mathbb E \left[\exp{(-\tilde xT_A)}\right]$$ tends to $1$ as $A \to
\infty$.

Consider next for any $\beta$ the exponential martingale
$\exp\left( \frac{\beta}{2} {\mathcal B}_{2t} - \frac{\beta^2}{4} t
\right)$, such that
$$\mathbb E \left[\exp\left( \frac{\beta}{2} {\mathcal B}_{2t} - \frac{\beta^2}{4} t\right)\right]
=\mathbb E \left[\exp\left( \frac{\beta}{2} {\mathcal
B}_0\right)\right]=1.$$
{As before, for $\beta \geq 0$, the martingale stays bounded from above by a fixed constant for $t\in [0,T_A]$
with $T_A<\infty$ a.s.}
One thus applies this exponential martingale at the stopping time $T_A$:
$$\mathbb E \left[\exp\left( \frac{\beta}{2} {\mathcal B}_{2T_A} - \frac{\beta^2}{4} T_A\right)\right]
=1.$$ By definition ${\mathcal B}_{2T_A}=2A-aT_A$. One thus gets the
identity
$${\mathbb E} \exp[- (\beta a/2+\beta^2/4) T_A]=\exp(-\beta A),$$
and it now suffices to identify $2\tilde x:=\beta a+\beta^2/2$, with $\beta \geq 0$, to
obtain the boundary KPZ with $\tilde \Delta:=\beta/\gamma$, and
$${\mathbb E}  \exp(- \tilde x T_A)=\delta^{\tilde \Delta}
=\exp(-\beta A)=\exp\left\{- A[ (a^2+4\tilde x)^{1/2} - a
]\right\},$$
in complete analogy to \eqref{e.expmartingalekpz}. \qed\end{proof}

The reader may observe that the boundary measures described above
 are preserved under the transformations described in
Proposition \ref{Qtransformation}.  One can use this to define the
boundary measure on more general domains, which may not have
piecewise linear boundary conditions.

We also remark that a similar procedure to that above allows us to
make sense of measure restricted to lines in the interior of the
domain.

\section{Discrete random surface dimensions and heuristics}
\label{discreteheuristicoverview}

Historically, one of the uses of the KPZ formula has been to make
heuristic predictions about the scaling exponents of random fractal
subsets of the plane (see, e.g., \cite{MR1723364, MR1749396,
MR2112128, 2000PhRvL..84.1363D, 2006math.ph...8053D}, and the
references surveyed therein for much more detail).

In this subsection, we give a very rough and very brief sketch of
what such a heuristic might entail in a simple example.  Readers
familiar with discrete quantum gravity models (a.k.a.\ random planar
map models, random quadrangulation models, etc.) should note that
these models have natural interpretations as continuum random metric
spaces as well. For example, a random planar quadrangulation $M_n$
on the sphere
--- chosen uniformly from the set of all simply connected planar
quadrangulations with $n$ quadrilaterals
--- can be viewed as a manifold by endowing each
quadrilateral with the metric of a unit square. (Of course, the
resulting manifold will have singularities: negative curvature point
masses at vertices where more than four unit squares coincide and
positive curvature point masses at vertices where fewer than four
unit squares coincide.)  We may then choose a uniform square from
among this set.  Taking an ``infinite volume limit'' (as $n \to
\infty$) one obtains an infinite random quadrangulation $M_\infty$
with a distinguished square.  (See, e.g., \cite{MR2013797} for a
precise description of this construction for triangulations.) This
infinite random surface can be conformally mapped to the plane in
such a way that the center of the distinguished square is mapped to
the origin and the volume of the image of the distinguished square
is a constant $\delta$ (with a rotation chosen uniformly at random).
The images of the unit squares of $M_\infty$ form a tiling of $\C$
by ``conformally distorted'' unit squares. Different squares have
different sizes with respect to the Euclidean metric on the plane;
intuitively, one would expect such a tiling to look something
vaguely like the tilings in Figures \ref{f.QG1}, \ref{f.QG2}, and
\ref{f.QG3} except that the ``squares'' would be randomly oriented
and distorted.  The pullback of the intrinsic metric of $M_\infty$
to the plane via this map takes the form $e^{\lambda}(dx^2 + dy^2)$
for some function random $\lambda$ (which has logarithmic
singularities at the images of the vertices of the squares).
Although the equivalence of Liouville quantum gravity and discrete
quantum gravity is taken as an Ansatz throughout much of the
literature, to our knowledge the following is the first precise
conjecture for the complete scaling limit of a discrete quantum
gravity model:

\begin{conjecture} \label{c.scalingconjecture}
As $\delta \to 0$, the function $\lambda$ converges in law (e.g.,
w.r.t. to the weak topology on the space of distributions on the
plane modulo additive constants) to $\gamma (h(\cdot) - \gamma \log|
\cdot|)$ where $h$ is an instance of the whole plane Gaussian free
field (defined up to additive constants) and $\gamma^2 = \kappa =
8/3$.
\end{conjecture}

We further conjecture that other values of $\gamma$ are obtained by
choosing a random quadrangulation together with a statistical
physical model on the quadrangulation (FK cluster model,
percolation, $O(N)$ model, uniform spanning tree); in this case, the
probability of a given quadrangulation is proportional to the
partition function of the statistical physics model on that
quadrangulation. (See the references on random matrix theory and
geometrical models cited in the introduction for much more detail;
see \cite{2006math.ph...8053D} for a review with additional
references.) One can also consider scaling limits on spheres or
higher genus surfaces, as well as different kinds of marked points
(corresponding to different logarithmic singularities in the scaling
limit); however, these are a bit more complicated to describe, so we
limit attention to the infinite volume case for now.

By the usual conformal invariance Ansatz, it is natural to expect
that if one conditions on the infinite quadrangulation, and then
samples the loops or trees in these models (as mapped into the
plane), their law (in the scaling limit) will be {\em independent}
of the metric.

Now suppose that for each $n$ we define a random subset $X_n$ of
$M_n$ (for example, $X_n$ could be the set of the squares hit by a
simple random walk started at the root square and stopped the first
time that the walk hits a square on the boundary of the
quadrangulation).  Then one can define a discrete scaling exponent
(analogous to the box counting exponent in (\ref{e.qse}), with
$\delta$ replaced by $n^{-1}$) as follows:
$$\Delta_D = \lim_{n \to \infty} \frac {\log\mathbb E(n^{-1} |X_n|)}{\log n^{-1}}.$$  Identifying
$X_n$ with its image in a conformal map to, say, $\D$, one might
guess that the random pair $(X_n, \lambda_n)$ --- where $e^{\lambda_n(z)}dz$ is uniform
measure on the discrete surface, mapped to $\D$ --- has a scaling limit
$(X, \lambda)$, where $X$ is a random subset of $\D$ (in our
example, it might be a Brownian motion) and $\lambda$ is some form
of the Gaussian free field.

If this is the case, then on a heuristic level, one would expect
that the quantum scaling exponent of $X$ is $\Delta=\Delta_D$,
since, in the notation of Corollary \ref{c.strongquantumKPZ2}, if we
write $\delta = n^{-1}$, we would expect that $\mathbb E[\delta
N(\mu, \delta, X)]$ scales like $\mathbb E (n^{-1} |X_n|)$.

In discrete quantum gravity models, it is often possible to compute
$\Delta_D$ explicitly (and rigorously) using random matrix
techniques or tree bijections; it is also often possible to compute
$\gamma$ directly using discrete quantum gravity machinery and so
heuristically obtain its value in the continuum limit.

Assuming values for $\Delta_D$ and $\gamma$ --- and assuming $\Delta
= \Delta_D$ --- the KPZ formula gives the Euclidean scaling
dimension of $X$.  In many interesting examples, $X$ is a random
fractal (a Schramm-Loewner evolution, for example, or the outer
boundary of a planar Brownian motion) whose Euclidean scaling
dimension might not be immediately obvious otherwise.

Finally, we mention that, in the standard realm of conformal field
theory, there exists a precise relation between the central charge
$c \leq 1$ of the statistical model coupled to quantum gravity and
the value of Liouville parameter,
$\gamma=\left(\sqrt{25-c}-\sqrt{1-c}\right)/\sqrt 6$,
\cite{MR947880,MR981529,MR1005268,1990PThPS.102..319S,Ginsparg-Moore},
as well as a corresponding connection between \SLEk/ and
Liouville quantum gravity models with $\gamma = \sqrt{\min\{\kappa,
16/\kappa \}}$.

Our result extends the validity of the KPZ relation outside that CFT
framework to any value of Liouville parameter $\gamma <2$, with the
Ansatz that the fractal set $X$ and the GFF are sampled
independently. A possible interpretation of the KPZ relation in that
case would be that it describes the quantum geometry of the given
fractal in the {\it quenched} random surface generated by random
graphs, equilibrated with a conformally invariant system with a
value of  $c$ or $\kappa$ corresponding to the chosen value of
$\gamma$.  For example, one could first choose a random graph
weighted by the critical Ising model partition function; and then
perform a loop erased random walk on that graph, ignoring Ising
clusters.  In this case, one would expect the Euclidean dimension of
the path to be that of \SLEkk2/ (which corresponds to loop erased
random walk), while the value of $\gamma$ describing the metric
would be $\sqrt{3}$ (which corresponds to the critical Ising model),
and one could use KPZ to predict the quantum scaling dimension.

Similar ideas appeared in previous numerical work
\cite{1998cond.mat..4137A,1999PhLB..460..271J}, but the data so far
appear as inconclusive.

Finally, we remark that the original (still accessible) arXiv version of this paper contained an additional section: a three-page sketch of some work in progress, including some results about the conformal welding of quantum random surfaces and about the scaling limits of discrete quantum gravity models.  Many of these results will appear in \cite{Zipper,DS3}.

\bigskip

\bigskip

\bigskip

\bigskip

{\bf Acknowledgments:}  We thank the IAS/Park City Mathematics
Institute, where this work was initiated, the School of Mathematics
of the Insitute for Advanced Study for its gracious hospitality in
successive stays during which this work was done, the ICTP in
Trieste, the \'Ecole de physique des Houches, the Centre de
recherches math\'ematiques de l'Universit\'e de Montr\'eal
and the Center for Theoretical Physics at MIT where
this work could be completed. B.D. wishes especially to thank
Michael and Marta Aizenman, Tom Spencer, Lynne Breslin, and Bob Jaffe for their generous
hospitality in Princeton, New York, and Cambridge. It is also a pleasure to thank Jacques
Franchi for suggesting the exponential martingale argument used in
Section \ref{expmart}.  We also thank Omer Angel, Peter Jones, Greg
Lawler, Andrei Okounkov, and Oded Schramm for stimulating
conversations and email correspondence, and Tom Alberts for comments
on a prior draft of this paper. The authors also thank Yuval Peres and the Theory Group of Microsoft Research for their
hospitality at the occasion of Oded Schramm's Memorial Conference. This research has made use of NASA's Astrophysics Data System.

\bibliographystyle{halpha}
\bibliography{ldpkpzrev}
\end{document}